\documentclass[11pt,a4paper]{article}%
\usepackage{amssymb}
\usepackage{amsfonts}
\usepackage{amsmath}
\usepackage[left=2cm,right=2cm,top=2cm,bottom=2.5cm]{geometry}
\usepackage{graphicx}
\usepackage[latin1]{inputenc}
\usepackage{amsthm}
\providecommand{\U}[1]{\protect\rule{.1in}{.1in}}
\providecommand{\lev}{\operatorname*{lev}}
\providecommand{\UMD}{\operatorname*{UMD}}

\providecommand{\NS}{\{0,1\}^2\backslash\{(0,0)\}}

\providecommand{\BMO}{\operatorname*{BMO}}
\providecommand{\parent}{\operatorname*{pre}}

\providecommand{\com}[1]{\complement#1}

\providecommand{\E}{\mathbb{E}}
\newtheorem{theorem}{Theorem}[section]

\newtheorem{defi}[theorem]{Definition}

\newtheorem{lemma}[theorem]{Lemma}

\newtheorem{proposition}[theorem]{Proposition}
\newenvironment{definition}{\begin{defi}\upshape}{\end{defi}}
\newenvironment{remark}{\noindent\textbf{Remark.}\upshape}{}

\numberwithin{equation}{section}
\begin{document}

\title{A Decomposition Theorem for Singular Integral Operators on Spaces of
Homogeneous Type}
\maketitle
\small
\begin{tabular}{p{10cm}l}
\textsc{Paul F.X. Müller} & \textsc{Markus Passenbrunner} \\
\textsc{Department of Analysis} & \textsc{Department of Analysis} \\ 
\textsc{J. Kepler University} & \textsc{J. Kepler University}\\
\textsc{Altenberger Strasse 69} & \textsc{Altenberger Strasse 69}\\
\textsc{A-4040 Linz} & \textsc{A-4040 Linz}\\
\textsc{Austria} & \textsc{Austria}\\
pfxm@bayou.uni-linz.ac.at & passenbr@bayou.uni-linz.ac.at
\end{tabular}
\normalsize
\tableofcontents
\newpage

\begin{abstract}
Let $(X,d,\mu)$ be a space of homogeneous type and $E$ a $\UMD$ Banach space. Under the assumption $\mu(\{x\})=0$ for all $x\in X$, we prove a decomposition theorem for singular integral operators on $(X,d,\mu)$ as a series of simple shifts and rearrangements plus two paraproducts. This gives a $T(1)$ Theorem in this setting.
\end{abstract}

\section{Introduction}
The $T(1)$-Theorem for scalar valued singular integral operators on $\mathbb{R}^n$ was initially proved by David and Journé (\cite{DavidJourne1984}) using Fourier analysis methods. It was later extended to spaces of homogeneous type by Coifman (unpublished, see \cite{Christ1990book} and \cite{CoifmanJonesSemmes1989}).
The structural framework for both proofs is given by Cotlar-Stein theorem on almost orthogonal operators.
Consequently, different methods had to be developed to obtain a $T(1)$ theorem for integral operators taking values in general Banach spaces. This was done by T. Figiel (\cite{Figiel1988} and \cite{Figiel1990}) who introduced a general method of decomposing integral operators into series of basic building blocks. This decomposition arises canonically by expanding the integral kernel along the isotropic Haar system in $\mathbb{R}^n\times \mathbb{R}^n$. Thus proving boundedness of integral operators is reduced to the following problems:
\begin{itemize}
\item Verify a priori norm estimates for the building blocks (this is independent of the underlying integral kernel).
\item Verify compensating coefficient estimates arising in the isotropic series expansion of the kernel (the decay of the coefficients depends on the size and smoothness of the kernel under investigation). 
\end{itemize}
The basic building blocks isolated by Figiel are simple rearrangements and shifts plus two paraproducts. These rearrangements and shifts act on the Haar system in $\mathbb{R}$. It is important to note that their definition depends expressly on the group structure of the underlying domain $(\mathbb{R}^n,+)$.
Figiel's decomposition was applied later to several singular integral operators beyond the Calderón-Zygmund class. These included applications to Dirichlet kernels of generalized Franklin systems (\cite{KamontMueller2006}) and interpolatory estimates arising in the theory of compensated compactness (\cite{LeeMuellerMueller2011}).

In the present paper 
we extend Figiel's decomposition method to the setting of spaces of homogeneous type. 
Our extension of this method 
is based on constructing -- without recourse to group structure -- a suitable class of rearrangement and shift operators that allow us to decompose singular integral operators on $(X,d,\mu)$ into a series of basic building blocks that can be analyzed and estimated by combinatorial means. The central result of this paper is the convergence of this operator-series (\ref{eq:decomp2}).

A source of renewed interest in spaces of homogeneous type is the recent development of diffusion wavelets and their multiresolution analysis that was carried out on spaces of homogeneous type by Coifman and Maggioni (\cite{CoifmanMaggioni2006}). We recall further that the vector-valued $T(1)$ theorem on spaces of homogeneous type is an essential first step towards the solution of the open classification problem for the vector valued Banach spaces $H^1_E(X,d,\mu)$. See \cite{Mueller1995} and \cite{MuellerSchechtman1991}.

\section{Martingale Preliminaries}\label{sec:prel}
In this section, we collect a set of martingale inequalities we use throughout the paper.
\subsection{Kahane's Contraction Principle}
We use Kahane's contraction principle in the following form (\cite{Kahane1985},\cite{MarcusPisier1981}).
\begin{theorem} 
[Kahane, contraction principle]\label{th:kahane}Let $e_{1},\ldots,e_{m}$ be
elements in a Banach space $E$ and $r_{1},\ldots,r_{m}$ be independent
Rademacher functions. If $a_{1},\ldots,a_{m}$ are real numbers with
$\sup_{k\leq m}|a_{k}|\leq1,$ we have for any $1\leq p<\infty$%
\[
\int_{0}^{1}\left\Vert \sum_{k=1}^{m}a_{k}r_{k}(t)e_{k}\right\Vert _{E}%
^{p}dt\leq\int_{0}^{1}\left\Vert \sum_{k=1}^{m}r_{k}(t)e_{k}\right\Vert
_{E}^{p}dt.
\]
\end{theorem}


\subsection{$\UMD$ spaces}
\label{sec:SteinBourgain}
\begin{definition}
A Banach space $E$ is called a \emph{$\UMD$--space}\emph{\ (unconditional for martingale
differences)}, if for every $1<p<\infty$ there exists a constant $\beta_{p}$
such that for every $E$-valued martingale difference sequence $(d_{k}%
)_{k\geq0}$ we have the inequality%
\begin{equation}
\left\Vert \sum_{k=0}^{n}\varepsilon_{k}d_{k}\right\Vert _{L^p_E}%
\leq\beta_{p}\left\Vert \sum_{k=0}^{n}d_{k}\right\Vert _{L^p_E}\label{eq:UMD}%
\end{equation}
for all sequences $\varepsilon$ of numbers in $\{-1,1\}$ and all $n\in\mathbb{N}$.
\end{definition}

\begin{remark}\\
1. We remark that if there exists one $1<p<\infty$ with a constant
$\beta_{p}$ such that $(\ref{eq:UMD})$ holds, we have automatically that for
all $1<p<\infty$ there exists a constant $\beta_{p}$ with $(\ref{eq:UMD}).$\\
2. Hilbert spaces are $\UMD$--spaces, $\UMD$--spaces are reflexive and the $\UMD$--property is a self dual isomorphic invariant (see for instance \cite{Figiel1988},\cite{Figiel1990}, \cite{FigielWojtaszczyk2001} or \cite{Burkholder2001}).
\end{remark}

\subsection{The space $\BMO$}
We let $(X,\mathcal{F},\mu)$ be a probability space and $\{\mathcal{F}_k\}_{k\in\mathbb{N}_0}$ be a sequence of $\sigma$-algebras such that $\mathcal{F}$ is generated by the union $\cup_k\mathcal{F}_k$. 
For $f\in L^1(X)$ we introduce the abbreviations
\[
\mathbb{E}_{k}f:=\mathbb{E}(f|\mathcal{F}_{k})\quad\text{and\quad}\Delta
_{k}:=\mathbb{E}_{k}-\mathbb{E}_{k-1}.
\]

\begin{definition}(Bounded Mean Oscillation).
A function $f:X\rightarrow \mathbb{R}$ is said to be in $\BMO(X,(\mathcal{F}_{k}))$
if and only if $f$ is in $L^{2}(X)$ and 
\begin{equation}
\left\|f\right\|_{\BMO}:=\sup_{k\in\mathbb{N}}\left\|\sqrt{\E_k(|f-\E_{k-1}f|^2)}\right\|_\infty<\infty. \label{eq:BMOdef}
\end{equation}
This is a norm, if we factor out the constants.
\end{definition}

\begin{remark}
\label{rem:BMO}
Recall that no matter what exponent $1\leq p<\infty$ in
$(\ref{eq:BMOdef})$ is chosen instead of $2$, the definition leads to the same
space $\BMO
(X,(\mathcal{F}_{k}))$ with equivalent norms (cf.
\cite{Garsia1973} or \cite{Bourgain1983}).
\end{remark}

\section{Extracting Rearrangements on Spaces of Homogeneous Type}\label{sec:homtype}
This section contains an extensive combinatorial analysis of dyadic cubes in spaces of homogeneous type. We recall first basic properties of those cubes and of the martingale differences they generate. Thus we construct orthonormal bases in $L^2(X)$ and $L^2(X\times X)$. 
Next we introduce a coloring on the collection of all dyadic cubes, so that on each monochromatic subcollection there are well defined rearrangement operators that act like "shifts by $q^m$ units" (Proposition \ref{prop:boundednumber}). The complications in the proof of this proposition are due to the fact that we need to have good quantitative control on the numbers of colors involved. This in turn is dictated by the nature of the kernel operators we treat in Section \ref{sec:main}. Theorem \ref{th:zerl} is the second main result of this section. It provides the combinatorial basis for the norm estimates of the rearrangement operators defined in Section \ref{sec:rearr}.
\subsection{Definitions}
\begin{definition}
Let $X$ be a set. A mapping $d:X\times X\rightarrow\mathbb{R}_{0}^{+}$ with
the properties\vspace{-0.2cm}
\begin{enumerate}
\item $d(x,y)=0\Leftrightarrow x=y$,\vspace{-0.2cm}
\item $d(x,y)=d(y,x)$,\vspace{-0.2cm}
\item $d(x,y)\leq K(d(x,z)+d(z,y))$ for all $x,y,z\in X$ and some constant $K\geq 1$ that is independent of $x,y,z$. \vspace{-0.2cm}
\end{enumerate}
is called a \emph{quasimetric }and $(X,d)$ is called a \emph{quasimetric
space.}
\end{definition}
Given a quasimetric $d$, we define the ball centered at $x\in X$ with radius $r>0$ as
\[
B(x,r):=\{y\in X:d(x,y)<r\}.
\]
Additionally, a set $A\subset X$ is called \emph{open} if and only if for all $x\in X$ there exists $r>0$ such that $B(x,r)\subseteq A$.

\begin{definition}
\label{def:hom}Let $(X,d)$ be a quasimetric space such that every ball in the quasimetric $d$ is open and $\mu$ a Borel measure. If there is an
$A>0$ such that%
\[
 0<\mu(B(x,2r))\leq A\mu(B(x,r))<\infty,\quad \text{for all }x\in X\text{ and all } r>0,
\]
then $(X,d,\mu)$ is called a \emph{space of homogeneous type. }Additionally,
if there exist constants $b_{1},b_{2}$ such that%
\[
b_{1}r\leq\mu(B(x,r))\leq
b_{2}r
\]
for all $x\in X$ and all $r$ with $\mu(\{x\})<r<\mu(X)$, we call the space of homogeneous type $(X,d,\mu)$ $\emph{normal.}$
\end{definition}
\begin{remark}
We note that if $(X,d,\mu)$ is a space of homogeneous type, then for all
$\lambda>0$ there exists $A_{\lambda},$ such that%
\[
\mu(B(x,\lambda r))\leq A_{\lambda}\mu(B(x,r))\quad \text{for all } x\in X\text{ and all } r>0.
\]
Since for a given quasimetric space $(X,d)$, the balls in $X$ are not necessarily open, we added this condition to the definition. This is the case, if for instance one has a Hölder condition for $d$: There exists $C<\infty$ and $0<\beta <1$ such that for all $x,y,z\in X$ we have
\begin{equation}
|d(x,z)-d(y,z)|\leq Cd(x,y)^\beta \operatorname{max}\{d(x,z),d(y,z)\}^{1-\beta}.\label{eq:lipcond}
\end{equation}
In fact, Mac{\'{\i}}as and Segovia proved in \cite{MaciasSegovia1979} that for every space of homogeneous type there exists an equivalent quasimetric with the desired Hölder property. Here, a quasimetric $d'$ is equivalent to a quasimetric $d$ if there exists a finite constant $C$ such that
\[
\frac{1}{C}d(x,y)\leq d'(x,y)\leq Cd(x,y),
\]
whenever $x,y\in X$.
\end{remark}
\paragraph{Standard assumptions on $X$:}
In the following, we always assume that the spaces $X$ we work with are spaces
of homogeneous type, equipped with a quasimetric $d$ and a Borel probability measure
$\mu.$ Additionally we impose the restriction that $X$ is normal and that for all $x\in X$ we have $\mu(\{x\})=0,$ i.e. we have no isolated points.

\subsection{Dyadic Cubes}\label{sec:cubes}
In a space of homogeneous type there are analogues for dyadic cubes in $\mathbb{R}^n$ (see \cite{Christ1990}\ and \cite{David1991}).

\begin{theorem}
\label{th:dyadic}Let $(X,d,\mu)$ be a space of homogeneous type. Then there
exist a system of open sets%
\[
\mathcal{A}:=\{Q_{\alpha}^{n}\subseteq X~|~n\in\mathbb{Z,\alpha\in}%
\mathcal{K}_{n}\},
\]
points $z_{\alpha}^{n}\in Q_{\alpha}^{n}$ and constants $q>1,$ $c_{1}%
,c_{2},c_{3},\eta\in\mathbb{R}^{+},N\in \mathbb{N}$ such that we have the following properties\vspace{-0.2cm}

\begin{enumerate}
\item For all $n\in\mathbb{Z}$ we have that $X=%
{\displaystyle\bigcup\limits_{\alpha\in\mathcal{K}_{n}}}
Q_{\alpha}^{n}$ up to $\mu$-null sets.\vspace{-0.2cm}

\item For $Q_{\alpha}^{m},Q_{\beta}^{n}$ with $m\leq n$ and $\alpha
\in\mathcal{K}_{m}$ and $\beta\in\mathcal{K}_{n}$ we have either $Q_{\alpha
}^{m}\subseteq Q_{\beta}^{n}$ or $Q_{\alpha}^{m}\cap Q_{\beta}^{n}=\emptyset.$ That means that the cubes $\{Q_{\alpha}^n\}$ are nested. \vspace{-0.2cm}

\item For each $Q_{\alpha}^{n}$ and every $m\geq n$ there is exactly one
$\beta\in\mathcal{K}_{m}$ such that $Q_{\alpha}^{n}\subseteq Q_{\beta}^{m}.$\vspace{-0.2cm}

\item For all $n\in\mathbb{Z}$ and for all $\alpha\in\mathcal{K}_{n}$ we have that $B(z_{\alpha}%
^{n},c_{1}q^{n})\subseteq Q_{\alpha}^{n}\subseteq B(z_{\alpha}^{n},c_{2}%
q^{n})$.\vspace{-0.2cm}

\item With
\[
\partial_{t}Q_{\alpha}^{n}:=\{x\in Q_{\alpha}^{n}:d(x,X\backslash Q_{\alpha
}^{n})\leq tq^{n}\},
\]
we have
\[
\forall n\in\mathbb{Z}\forall\alpha\in\mathcal{K}_{n}:\mu(\partial
_{t}Q_{\alpha}^{n})<c_{3}t^{\eta}\mu(Q_{\alpha}^{n}).
\]\vspace{-0.4cm}

\item For all $n\in\mathbb{Z,}$ the set $\mathcal{K}_{n}$ is countable.\vspace{-0.2cm}

\item For all $n\in\mathbb{Z}$ and
all $\alpha\in\mathcal{K}_{n}$ we have $|\{\beta\in\mathcal{K}_{n-1}:Q_{\beta
}^{n-1}\subseteq Q_{\alpha}^{n}\}|\leq N$.\vspace{-0.2cm}

\item For all $n\in\mathbb{Z,}$ $\alpha\in\mathcal{K}_{n}$ there is a subset
$E$ of $\mathcal{K}_{n-1}$ with $|E|\leq N$ such that%
\[
Q_{\alpha}^{n}=%
{\displaystyle\bigcup\limits_{\beta\in E}}
Q_{\beta}^{n-1}\quad\text{up to }\mu\text{-null sets.}
\]\vspace{-0.2cm}
\end{enumerate}\vspace{-0.5cm}
\end{theorem}

\begin{remark}
We note that these dyadic cubes were constructed by Christ in \cite{Christ1990} and by David in \cite{David1991} in a slightly different way. 
We further remark that in the future use of the dyadic cubes, we neglect $\mu$-null sets in points $1$ and $8$ of Theorem $\ref{th:dyadic}$ and assume equality.
\end{remark}\vspace{0.2cm}

We now collect a few useful definitions, which we will need in the sequel.
\begin{definition}\label{def:misc} We let
\[
\mathcal{A}_{n}:=\{Q_{\alpha}^{n}:\alpha\in\mathcal{K}_{n}\},
\]
be the set of dyadic cubes with level $n\in\mathbb{Z}.$ Furthermore, let $A\in\mathcal{A}_{n+1}$ and
choose $A^{\ast}(A)\in\mathcal{A}_{n}$ arbitrarily (but fixed for all subsequent sections) with $A^{\ast}(A)\subseteq A.$ Then we
set%
\[
\mathcal{E}(A):=\{B\in\mathcal{A}_{n}:B\subseteq A\backslash A^{\ast}(A)\}.
\]
We denote the cardinality $|\mathcal{E}(A)|$ of this set by $N(A)$.
Additionally, we define the level of $A\in\mathcal{A}_{n+1}$ as
\[
\lev A:=n+1.
\]
The unique element $A\in\mathcal{A}_{n+1}$ such that for $Q\in\mathcal{A}_n$ we have $Q\subset A$ will be denoted by
\begin{equation}
\parent Q, \label{eq:defparent}
\end{equation}
which indicates that $A$ is the predecessor of $Q$.
Furthermore, we define a subset of dyadic cubes
\[
\mathcal{E}(\mathcal{A}):=\bigcup_{A\in\mathcal{A}}\mathcal{E}(A)
\]

\end{definition}
\begin{remark}
Due to Point 7 of Theorem \ref{th:dyadic} we have that the cardinality $N(A)$ of $\mathcal{E}(A)$ is bounded by a uniform constant $N-1$ independent of $A\in\mathcal{A}%
_{n+1}$.
\end{remark}

\subsection{Martingale Differences}

Let $(X,d,\mu)$ be a space of homogeneous type with $\mu(X)=1$. Then we have $X=Q_{1}^{0},$
$\mathcal{K}_{0}=\{1\},\mathcal{A}_{0}=\{X\}$ and $\mathcal{K}_{n}=\emptyset$
for all $n\in\mathbb{N}.$ We use then dyadic cubes to build an orthonormal basis in
$L^{2}(X,d,\mu)$ consisting of martingale differences. Fix $n\in-\mathbb{N}$,
$A\in\mathcal{A}_{n+1}$ and enumerate the elements in $\mathcal{E}(A)$ in the
way that $\mathcal{E}(A)=\{Q_{1},\ldots,Q_{N(A)}\}.$ Additionally we set
$Q_{N(A)+1}:=A^{\ast}(A).$ We define the following functions, supported on
$A.$

\begin{definition}
\label{def:basisfunct}We define for $1\leq k\leq N(A)$ and $x\in X$%
\[
d_{Q_{k}}(x):=c_{Q_{k}}\left\{
\begin{array}
[c]{ll}%
0, & \text{if }x\in%
{\displaystyle\bigcup\limits_{j=1}^{k-1}}
Q_{j}\cup(X\backslash A)\\
\sum_{j=k+1}^{N(A)+1}\mu(Q_{j}), & \text{if }x\in Q_{k}\\
-\mu(Q_{k}), & \text{if }x\in%
{\displaystyle\bigcup\limits_{j=k+1}^{N(A)+1}}
Q_{j}%
\end{array}
,\right.
\]
where we choose $c_{Q_{k}}$ such that
\begin{equation}
\left\Vert d_{Q_{k}}\right\Vert _{2}=1.\label{eq:normalization}%
\end{equation}
\end{definition}

\begin{remark}
The functions defined in Definition \ref{def:basisfunct} are obviously a martingale difference sequence. We record here also that these martingale differences are just the result of the Gram Schmidt orthogonalization process applied to the indicator functions
\begin{equation}
1_{A},1_{Q_{1}},\ldots,1_{Q_{N(A)}}.\label{eq:indicator}
\end{equation}
\end{remark}

\begin{figure}
\centering
\ifpdf
  \setlength{\unitlength}{1bp}%
  \begin{picture}(347.28, 206.70)(0,0)
  \put(0,0){\includegraphics{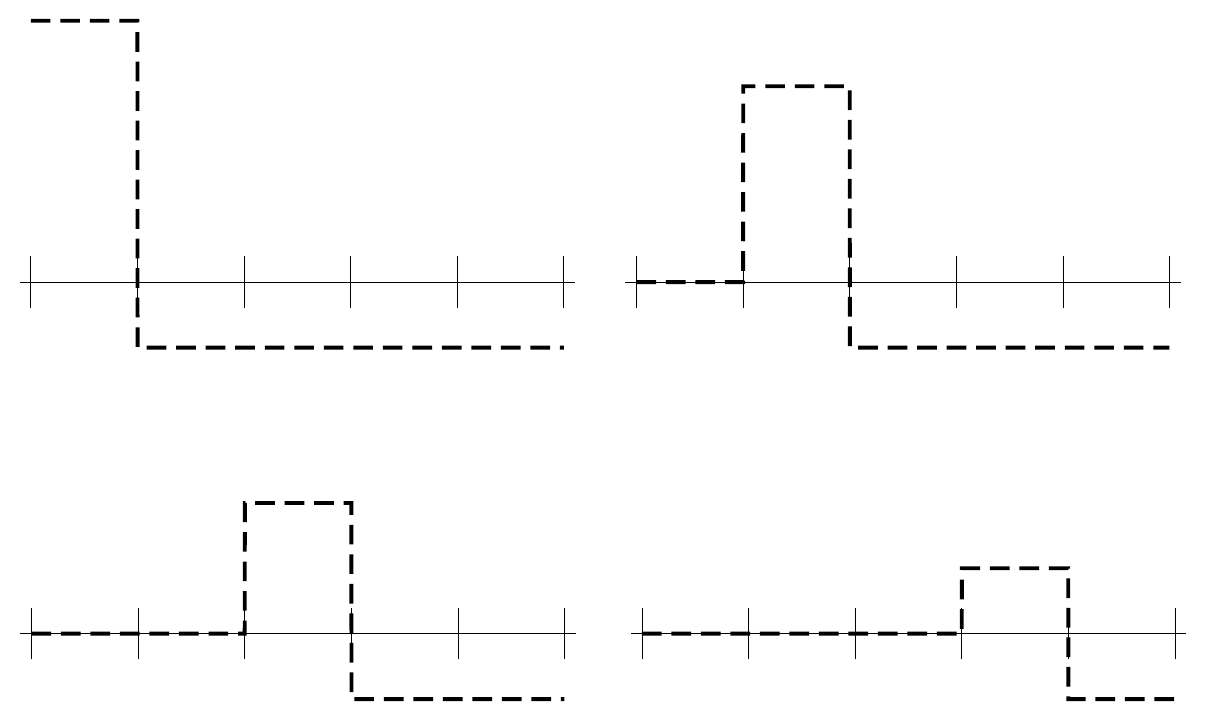}}
  \put(78.49,177.16){\fontsize{9.96}{11.95}\selectfont $d_{Q_1}$}
  \put(255.41,67.31){\fontsize{9.96}{11.95}\selectfont $d_{Q_4}$}
  \put(78.49,67.31){\fontsize{9.96}{11.95}\selectfont $d_{Q_3}$}
  \put(255.41,176.78){\fontsize{9.96}{11.95}\selectfont $d_{Q_2}$}
  \put(17.97,115.88){\fontsize{9.96}{11.95}\selectfont $Q_1$}
  \put(48.42,115.88){\fontsize{9.96}{11.95}\selectfont $Q_2$}
  \put(79.26,115.88){\fontsize{9.96}{11.95}\selectfont $Q_3$}
  \put(109.71,115.88){\fontsize{9.96}{11.95}\selectfont $Q_4$}
  \put(140.55,115.88){\fontsize{9.96}{11.95}\selectfont $Q_5$}
  \put(195.28,14.89){\fontsize{9.96}{11.95}\selectfont $Q_1$}
  \put(225.73,14.89){\fontsize{9.96}{11.95}\selectfont $Q_2$}
  \put(256.57,14.89){\fontsize{9.96}{11.95}\selectfont $Q_3$}
  \put(287.02,14.89){\fontsize{9.96}{11.95}\selectfont $Q_4$}
  \put(317.85,14.89){\fontsize{9.96}{11.95}\selectfont $Q_5$}
  \put(17.97,14.89){\fontsize{9.96}{11.95}\selectfont $Q_1$}
  \put(48.42,14.89){\fontsize{9.96}{11.95}\selectfont $Q_2$}
  \put(79.26,14.89){\fontsize{9.96}{11.95}\selectfont $Q_3$}
  \put(109.71,14.89){\fontsize{9.96}{11.95}\selectfont $Q_4$}
  \put(140.55,14.89){\fontsize{9.96}{11.95}\selectfont $Q_5$}
  \put(192.20,115.88){\fontsize{9.96}{11.95}\selectfont $Q_1$}
  \put(222.65,115.88){\fontsize{9.96}{11.95}\selectfont $Q_2$}
  \put(253.48,115.88){\fontsize{9.96}{11.95}\selectfont $Q_3$}
  \put(283.93,115.88){\fontsize{9.96}{11.95}\selectfont $Q_4$}
  \put(314.77,115.88){\fontsize{9.96}{11.95}\selectfont $Q_5$}
  \end{picture}%
\else
  \setlength{\unitlength}{1bp}%
  \begin{picture}(347.28, 206.70)(0,0)
  \put(0,0){\includegraphics{Basis.eps}}
  \put(78.49,177.16){\fontsize{9.96}{11.95}\selectfont $d_{Q_1}$}
  \put(255.41,67.31){\fontsize{9.96}{11.95}\selectfont $d_{Q_4}$}
  \put(78.49,67.31){\fontsize{9.96}{11.95}\selectfont $d_{Q_3}$}
  \put(255.41,176.78){\fontsize{9.96}{11.95}\selectfont $d_{Q_2}$}
  \put(17.97,115.88){\fontsize{9.96}{11.95}\selectfont $Q_1$}
  \put(48.42,115.88){\fontsize{9.96}{11.95}\selectfont $Q_2$}
  \put(79.26,115.88){\fontsize{9.96}{11.95}\selectfont $Q_3$}
  \put(109.71,115.88){\fontsize{9.96}{11.95}\selectfont $Q_4$}
  \put(140.55,115.88){\fontsize{9.96}{11.95}\selectfont $Q_5$}
  \put(195.28,14.89){\fontsize{9.96}{11.95}\selectfont $Q_1$}
  \put(225.73,14.89){\fontsize{9.96}{11.95}\selectfont $Q_2$}
  \put(256.57,14.89){\fontsize{9.96}{11.95}\selectfont $Q_3$}
  \put(287.02,14.89){\fontsize{9.96}{11.95}\selectfont $Q_4$}
  \put(317.85,14.89){\fontsize{9.96}{11.95}\selectfont $Q_5$}
  \put(17.97,14.89){\fontsize{9.96}{11.95}\selectfont $Q_1$}
  \put(48.42,14.89){\fontsize{9.96}{11.95}\selectfont $Q_2$}
  \put(79.26,14.89){\fontsize{9.96}{11.95}\selectfont $Q_3$}
  \put(109.71,14.89){\fontsize{9.96}{11.95}\selectfont $Q_4$}
  \put(140.55,14.89){\fontsize{9.96}{11.95}\selectfont $Q_5$}
  \put(192.20,115.88){\fontsize{9.96}{11.95}\selectfont $Q_1$}
  \put(222.65,115.88){\fontsize{9.96}{11.95}\selectfont $Q_2$}
  \put(253.48,115.88){\fontsize{9.96}{11.95}\selectfont $Q_3$}
  \put(283.93,115.88){\fontsize{9.96}{11.95}\selectfont $Q_4$}
  \put(314.77,115.88){\fontsize{9.96}{11.95}\selectfont $Q_5$}
  \end{picture}%
\fi
\caption{\label{fig:basisfunct}%
 Schematic plots of functions in Definition \ref{def:basisfunct}, where $N$ is set to $5$.}
\end{figure}

Now we enumerate all the functions $d_{Q},$ $Q\in\mathcal{E}(A)$ where
$A\in\mathcal{A}_{n+1},n\in-\mathbb{N}$ in a canonical way, we set
\[
d_{0}:=1_{X}%
\]
and get the functions that are a basis in the constant functions on
$\{Q_{1},\ldots,Q_{N(X)},A^{\ast}(X)\},$ where $Q_{i}\in\mathcal{E}(X)$ for
$1\leq i\leq N(X)$ and set%
\[
d_{1}=d_{Q_{1}},\ldots,d_{N(X)}=d_{Q_{N(X)}}.
\]
We continue with this procedure on every $Q_{i},$ so we get an enumeration of
all functions $d_{Q},$ $Q\in\mathcal{E}(A)$ for $A\in\mathcal{A}_{n+1},$
$n\in-\mathbb{N}$ such that the order is preserved in the following way
\[
k\leq j\Rightarrow \lev R\geq \lev Q\quad\text{ for } d_k=d_R\text{ and } d_j=d_Q.
\]
We refer to the functions $d_{Q}$ as Haar functions.
According to this enumeration we define $\sigma$-algebras:%
\[
\mathcal{F}_{i}:=\sigma(d_{0},\ldots,d_{i})\quad\text{for }i\in\mathbb{N}_0.
\]
With respect to this filtration, the collection $\{d_{k}\}_{k\in\mathbb{N}}$
is a martingale difference sequence, since we have $\mathbb{E}(d_{k}|\mathcal{F}_{k-1})=0$ for every $k\in\mathbb{N}$.
Another important sequence of $\sigma$-algebras that we need later is a suitable subsequence of the $\sigma$-algebras just created. We set
\begin{equation}
\mathcal{F}^{\lev}_k:=\sigma(\mathcal{A}_{-k})\quad\text{for }k\in\mathbb{N}_0,\label{eq:sigmalev}
\end{equation}
where the superscript $\lev$ should indicate that $\mathcal{F}^{\lev}_k$ is the $\sigma$-algebra generated by all dyadic cubes of level $-k$. 

As in the case $X=\mathbb{R}$ with the standard Haar functions, we have that the $L^\infty$ norm of an $L^2$ normalized Haar function $d_Q$ is (approximately) $\mu(Q)^{-1/2}$, which is a simple consequence of Theorem \ref{th:dyadic} and the normality of $X$.
\begin{lemma}\label{lem:supboundd}
There exists a constant $c<\infty$ depending only on $X$ such that 
\[
c^{-1}\mu(Q)^{-1/2}\leq\left\|d_Q\right\|_\infty\leq c\mu(Q)^{-1/2}\quad \text{for all }Q\in\mathcal{E}(\mathcal{A}).
\]

\end{lemma}

Another simple consequence of Theorem \ref{th:dyadic} is
\begin{lemma}
$\cup_{l=1}^{\infty}\mathcal{F}_{l} $ generates the Borel $\sigma$-algebra on X.
\end{lemma}


\begin{remark}
If $E$ is a $\UMD$--space, it is in particular reflexive and thus satisfies the Radon-Nikodym property. So, the martingale convergence theorem (see \cite{Chatterji1968}) and the above lemma yield that for $f\in L^p_E(X)$ we have
that%
\[
\lim_{k\rightarrow\infty}\left\Vert \mathbb{E}(f|\mathcal{F}_{k})-f\right\Vert
_{L_{E}^{p}(X)}=0
\]
for all $1\leq p<\infty.$ So we get for every $f\in L_{E}^{p}(X)$ a unique series expansion
\[
f=\sum_{k=0}^{\infty}a_{k}d_{k},\quad a_{k}\in E,
\]
which converges unconditionally in $L_{E}^{p}(X)$ for $1<p<\infty$. In particular for $p=2$ and
$E=\mathbb{R}$, $(d_{k})_{k\in\mathbb{N}}$ is an orthonormal basis.
\end{remark}

\subsection{Isotropic Basis in $L^{2}(X\times X)$}\label{sec:isotr}

Next we introduce an isotropic orthogonal basis in $L^{2}(X\times X).$ Here, the word isotropic means that for an element $f\otimes g$ of this basis (here, $f\otimes g(x,y):=f(x)g(y)$ is the standard tensor product of two functions), the support looks like a square and not like a rectangle.  Most of the notation used in the sequel was introduced in Definition \ref{def:misc}. 
Let $\varepsilon\in \{0,1\}$. For $Q\in \mathcal{E}(A)$ and $A\in \mathcal{A}$ we define
\[
d_Q^{(\varepsilon)}:=d_Q \quad\text{for }\varepsilon=1,\quad \text{ and }\quad d_Q^{(\varepsilon)}:=\frac{1_A}{\sqrt{\mu(A)}} \quad\text{for }\varepsilon=0.
\]
Note that the function $d_Q^{(0)}$ is $L^2$-normalized as is $d_Q^{(1)}$.
With these settings, we define the collection of functions on $X\times X$:
\begin{equation}
Z:=\{1_X\otimes 1_X\}\cup \{d_Q^{(\varepsilon_1)}\otimes d_R^{(\varepsilon_2)} :Q,R\in\mathcal{E}(\mathcal{A}),\lev Q=\lev R,\varepsilon=(\varepsilon_1,\varepsilon_2)\in \NS\}. \label{eq:orthz}
\end{equation}
Explicitly, up to constants, the three groups in ($\ref{eq:orthz}$) have the form
\begin{align}
\{d_{Q}\otimes d_{R}  & :A,B\in\mathcal{A}_{n+1},Q\in\mathcal{E}%
(A),R\in\mathcal{E}(B),n\in-\mathbb{N}\},\label{eq:basis1}\\
\{d_{Q}\otimes1_{B}  & :A,B\in\mathcal{A}_{n+1},Q\in\mathcal{E}(A),n\in
-\mathbb{N}\},\label{eq:basis2}\\
\{1_{A}\otimes d_{R}  & :A,B\in\mathcal{A}_{n+1},R\in\mathcal{E}%
(B),n\in-\mathbb{N}\}\label{eq:basis3}%
\end{align}
The system $Z$ forms an orthonormal basis in $L^2(X\times X)$ and this result follows from the well known classical 
\begin{lemma}
\label{lem:orthprod}If $\{e_{k}\}_{k=1}^{\infty}$ is an orthogonal basis in
$L^{2}(X,\mathcal{F},\mu),$ then $\{e_{k}\otimes e_{j}\}_{k,j=1}^{\infty}$ is
an orthogonal basis in $L^{2}(X\times X,\mathcal{F}\otimes\mathcal{F}%
,\mu\otimes\mu).$
\end{lemma}


\begin{lemma}\label{lem:basis}
$Z$ is an orthonormal basis in $L^{2}(X\times X).$
\end{lemma}

\begin{proof}
Since the verification of orthonormality is a straightforward calculation, we proceed with showing the basis property$\mathbb{.}$ Since
we know from Lemma \ref{lem:orthprod} that the set%
\[
\{d_{S}\otimes d_{T}:S\in\mathcal{A}_{m+1},T\in\mathcal{A}_{n+1}%
,n,m\in-\mathbb{N}\}\quad\text{with\quad}d_{X}:=1_{X}%
\]
is an orthogonal basis in $L^{2}(X\times X),$ we have to show that each
$d_{S}\otimes d_{T}$ can be decomposed in a finite linear combination of
functions of the form $(\ref{eq:basis1})-(\ref{eq:basis3}).$ To do that, we
need the following identities:%
\begin{align}
1_{U}  & =1_{A^{\ast}(U)}+\sum_{V\in\mathcal{E}(U)}1_{V},\quad U\in
\mathcal{A}_{m+1},\label{eq:id1}\\
d_{R}  & =c_{1}1_{A^{\ast}(B)}+\sum_{V\in\mathcal{E}(B)}c_{V}1_{V},\quad
R\in\mathcal{E}(B),c_{1},c_{V}\in\mathbb{R} \text{ for }V\in\mathcal{E}(B).\label{eq:id2}
\end{align}
We then have four cases:
\begin{enumerate}
\item Let $d_{S}=1_{X},$ $d_{T}=1_{X},$ then clearly $d_{S}\otimes d_{T}$ $\in
Z.$

\item $d_{S}=1_{X},$ $T\in\mathcal{A}_{n},n\in-\mathbb{N,}$ $B\in
\mathcal{A}_{n+1}$ with $T\in\mathcal{E}(B).$ Then we get recursively from
$(\ref{eq:id1}),$ that $1_{X}$ is a finite linear combination of functions of
the form $1_{C},$ where $C\in\mathcal{A}_{n+1}.$ With $(\ref{eq:basis3}),$ we
see that $1_{X}\otimes d_{T}\in\operatorname*{lin}Z.$

\item Analogously, we treat the case $d_{T}=1_{X}$ and $d_{S}\neq1_{X}.$

\item $S\in\mathcal{A}_{n},T\in\mathcal{A}_{m},m,n\in-\mathbb{N,}$
$S\in\mathcal{E}(A),T\in\mathcal{E}(B),A\in\mathcal{A}_{n+1},B\in
\mathcal{A}_{m+1}.$ If $m=n,$ we see from $(\ref{eq:basis1})$ that
$\ d_{S}\otimes d_{T}\in Z.$ Without loss of generality we now assume that
$m>n$ and we decompose $d_{T}$ in the form $(\ref{eq:id2}).$ Additionally, if
$m>n+1,$ we proceed recursively with $(\ref{eq:id1})$ and get from
$(\ref{eq:basis2})$\ that $d_{S}\otimes d_{T}\in\operatorname*{lin}Z.$ \qedhere
\end{enumerate}
\end{proof}

\subsection{Dyadic Annuli}\label{sec:dyadicannuli}

Recall that $\mathcal{A}_{n}$ is the set of dyadic cubes of level $n$ for $n\in -\mathbb{N}_0$ and $\mathcal{A}_0$ consists only of the whole space $X$ and the size of cubes \emph{decreases} with \emph{decreasing} index $n$.
We now introduce the set of all pairs of dyadic cubes of the same level%
\[
\mathcal{C}:=\{(A,B):A,B\in\mathcal{A}_{n},n\in -\mathbb{N}_0\}%
\]
and its decomposition into annuli $\mathcal{C=}
{\displaystyle\bigcup\limits_{m=0}^{\infty}}
\mathcal{C}_{m},$
where
\[
\mathcal{C}_{m}=\{(A,B)\in\mathcal{C}:q^{m-1+\operatorname*{lev}A}\leq
d(A,B)<q^{m+\operatorname*{lev}A}\}\quad\text{for }m\in\mathbb{N}%
\]
and%
\[
\mathcal{C}_{0}=\{(A,B)\in\mathcal{C}:d(A,B)<q^{\operatorname*{lev}A}\}.
\]
Recall that $\lev A$ denotes the level of $A$ (that is if $A\in\mathcal{A}_n$, then $\lev A=n$) and $q$ is the constant from Theorem \ref{th:dyadic} that determines the growth factor of the cubes in each level. This definition can be interpreted in
the following way: Given $A\in\mathcal{A}_{n},$ we draw an annulus around $A$
with inner radius $q^{m-1+\operatorname*{lev}A}$ and outer radius
$q^{m+\operatorname*{lev}A}$ and take all pairs $(A,B)$ such that $B$ has no
point inside the smaller circle and $B$ has at least one point inside the
larger circle. It is crucial that the annulus grows with the size of $A.$

%
%
%
%

\subsection{Extracting Rearrangements - Further Decomposition of Annuli}
The aim of this section is to extract (as few as possible) subcollections
$\mathcal{C}_{m,i}$ from $\mathcal{C}_{m}$ such that for each $(A,B)\in
\mathcal{C}_{m,i}$ we have that $B$ is uniquely determined by $A$ and $A\ $is
uniquely determined by $B.$ The benefit of this decomposition is that on
$\mathcal{C}_{m,i}$ we can define an injective mapping $\tau$ such that
$B=\tau(A)$ (see Definition \ref{def:rearr}). We start with the following observation:

\begin{lemma}
\label{lem:max}There exists a constant $M_{0}$ independent of $n$
and $m$, such that for $A\in\mathcal{A}_{n}$ there are at most
$M_{0}q^{m}$ elements $B\in\mathcal{A}_{n}$ with $(A,B)\in\mathcal{C}_{m}.$
\end{lemma}
So, roughly speaking, in an annulus of level $m$ around $A$, there are at most $q^m$ cubes of the same size as $A$.
This lemma is easily proved using the properties of dyadic cubes in Theorem \ref{th:dyadic} and the normality of $X$.

\begin{remark}
The same argument shows that for each $C>0$ there exists a constant
$M_{0}$ s.t. for $A\in\mathcal{A}_{n}$ we have at most $M_{0}$ elements
$B\in\mathcal{A}_{n}$ with%
\[
d(A,B)\leq Cq^{n}.\vspace{-0.5cm}
\]
\end{remark}

\begin{proposition}
\label{prop:boundednumber}Let $M_{1}:=2M_0$ with $M_0$ from Lemma \ref{lem:max}. Then we have for all $m\in\mathbb{N}_{0}$ that the collection $\mathcal{C}_{m}\subseteq\mathcal{A}%
\times\mathcal{A}$ admits a decomposition as%
\[
\mathcal{C}_{m}=\mathcal{C}_{m,1}\cup\cdots\cup\mathcal{C}_{m,M_1 q^m}
\]
so that each of the collections $\mathcal{C}_{m,i},$ $1\leq i\leq M_1 q^m$
satisfies the two conditions
\begin{enumerate}
\item For $B\in\mathcal{A}$ there exists at most one $A\in\mathcal{A}$ with
$(A,B)\in\mathcal{C}_{m,i}.$

\item For $A\in\mathcal{A}$ there exists at most one $B\in\mathcal{A}$ with
$(A,B)\in\mathcal{C}_{m,i}.$
\end{enumerate}
\end{proposition}
\begin{remark}
For the applications in Section \ref{sec:main} it is important that $\mathcal{C}_m$ is decomposed in $M_1 q^m$ subcollections (and not more). For instance the estimate $q^{2m}$ would be much simpler to obtain, but would not allow us to treat singular integral operators.
\end{remark}

\begin{proof}
\begin{description}
\item[Step 1:] Idea of the proof:\newline Let $Q\in\mathcal{A}$. Then we
define the ring collection of $Q:$%
\[
\mathcal{O}_{m}(Q):=\{R\in\mathcal{A}:(Q,R)\in\mathcal{C}_{m}\}.
\]
We will show that there exists $I(Q)\subseteq\{1,\ldots,M_{1}q^{m}\}=:I$ and
an enumeration of the dyadic cubes in $\mathcal{O}_{m}(Q)$ such that%
\[
\mathcal{O}_{m}(Q)=\{R_{i}(Q):i\in I(Q)\}
\]
and we have the following property:%
\begin{equation}
\forall Q,Q^{\prime}\in\mathcal{A},~Q\neq Q^{\prime}~\forall j\in I(Q)\cap I(Q^{\prime
}):R_{j}(Q)\neq R_{j}(Q^{\prime}).\label{eq:zerl2}%
\end{equation}
Then we can define the decomposition%
\[
\mathcal{A}_{m,i}=\{Q\in\mathcal{A}:i\in I(Q)\}\quad\text{and}\quad \mathcal{C}_{m,i}=\{(Q,R_{i}(Q)):Q\in\mathcal{A}_{m,i}\}.
\]
We thus obtain
\[
\mathcal{C}_{m}=\mathcal{C}_{m,1}\cup\cdots\cup\mathcal{C}_{m,M_1 q^m}%
\]
and the desired properties hold.

\item[Step 2:] Construction of the enumeration:\newline Let $\mathcal{A}%
=\{Q^{(k)}:k\in\mathbb{N\}}$ be an enumeration of all dyadic cubes$\mathbb{.}$
We proceed by induction over $k.$ For $k=1$ choose $I(Q^{(1)})=\{1,\ldots
,|\mathcal{O}_{m}(Q^{(1)})|\}$ and select any enumeration of the
cubes$\mathcal{\ O}_{m}(Q^{(1)}).$ Observe that with Lemma \ref{lem:max} we
have that $|\mathcal{O}_{m}(Q^{(1)})|\leq M_{0}q^{m}.$ Now let $k\in
\mathbb{N}$ and assume we have constructed%
\[
I(Q^{(1)}),\ldots,I(Q^{(k)})
\]
with%
\[
\mathcal{O}_{m}(Q^{(l)})=\{R_{i}(Q^{(l)}):i\in I(Q^{(l)})\}\text{ for }l\leq k
\]
such that the following holds%
\[
\forall Q,Q^{\prime}\in\{Q^{(1)},\ldots,Q^{(k)}\},~Q\neq Q^{\prime}~\forall j\in
I(Q)\cap I(Q^{\prime}):R_{j}(Q)\neq R_{j}(Q^{\prime}).
\]
We will now construct $I(Q^{(k+1)}).$ To do this we first set%
\[
\{R^{(1)},\ldots,R^{(M_{\ast})}\}=\mathcal{O}_{m}(Q^{(k+1)}),\quad\text{where
}M_{\ast}\leq M_{0}q^{m}.
\]

\item[Step 2a:] We start a second induction and begin with $R^{(1)}.$ We will
define the index $\operatorname*{ind}_{Q^{(k+1)}}R^{(1)}$ of $R^{(1)}$ in the
enumeration $\mathcal{O}_{m}(Q^{(k+1)})$ as follows. We put%
\[
V(R^{(1)})=\{Q^{\prime}\in\{Q^{(1)},\ldots,Q^{(k)}\}:R^{(1)}\in\mathcal{O}%
_{m}(Q^{\prime})\},
\]
so $V(R^{(1)})$ contains the cubes $Q^{\prime}$ for which $R^{(1)}$ is in
their ring collection $\mathcal{O}_{m}(Q^{\prime}).$ Now, since $V(R^{(1)}%
)\subseteq\mathcal{O}_{m}(R^{(1)}),$ we have an estimate for the cardinality
of $V(R^{(1)}):$%
\begin{equation}
|V(R^{(1)})|\leq M_{0}q^{m}.\label{eq:zerl1}%
\end{equation}
For $Q^{\prime}\in V(R^{(1)})$ we already defined the indices
$\operatorname*{ind}_{Q^{\prime}}R^{(1)}\in I.$ Next we let%
\[
L(R^{(1)})=\{\operatorname*{ind}_{Q^{\prime}}R^{(1)}:Q^{\prime}\in
V(R^{(1)})\}
\]
the indices of $R^{(1)}$ in the enumeration of $Q^{\prime}.$ According to
$(\ref{eq:zerl1})$, we have%
\[
|L(R^{(1)})|\leq M_{0}q^{m}%
\]
and $|I|=M_{1}q^m.$ For the reduced index set, defined as%
\[
I^{\text{red}}=I\backslash L(R^{(1)}),
\]
we have%
\[
|I^{\text{red}}|\geq M_{1}q^{m}-M_{0}q^{m}.
\]
In particular, we have $I^{\text{red}}%
\neq\emptyset.$ So we select any element in $I^{\text{red}}$ to be the index
of $R^{(1)}$ for $Q^{(k+1)}:$%
\[
\operatorname*{ind}_{Q^{(k+1)}}R^{(1)}\in I^{\text{red}}.
\]
Thus the beginning of the second induction is completed.

\item[Step 2b:] Next we fix $j<M_{\ast}\leq M_{0}q^{m}.$ We now assume that we
already defined%
\[
\operatorname*{ind}_{Q^{(k+1)}}R^{(1)},\ldots,\operatorname*{ind}_{Q^{(k+1)}%
}R^{(j)},
\]
so we pick $R^{(j+1)}\in\mathcal{O}_{m}(Q^{(k+1)}).$ As in the beginning of
the induction, we set
\[
V(R^{(j+1)})=\{Q^{\prime}\in\{Q^{(1)},\ldots,Q^{(k)}\}:R^{(j+1)}\in
\mathcal{O}_{m}(Q^{\prime})\}.
\]
We again have $V(R^{(j+1)})\subseteq\mathcal{O}_{m}(R^{(j+1)})$ and thus an
estimate for the cardinality%
\[
|V(R^{(j+1)})|\leq M_{0}q^{m}.
\]
Next let%
\[
L(R^{(j+1)})=\{\operatorname*{ind}_{Q^{\prime}}R^{(j+1)}:Q^{\prime}\in
V(R^{(j+1)})\}
\]
be the indices of $R^{(j+1)}$ in the enumeration of $Q^{\prime}.$ Since
$|L(R^{(j+1)})|\leq M_{0}q^{m},$ we have for the reduced index set%
\[
I^{\text{red}}=I\backslash(L(R^{(j+1)})\cup\{\operatorname*{ind}_{Q^{(k+1)}%
}R^{(1)},\dots,\operatorname*{ind}_{Q^{(k+1)}}R^{(j)}\})
\]
an estimate for the cardinality%
\[
|I^{\text{red}}|> M_{1}q^{m}-M_{0}q^{m}-M_{\ast}\geq(M_{1}-2M_{0})q^{m},
\]
so we have due to the definition of $M_1$ that $I^{\text{red}}\neq\emptyset.$ We
finally select then the index $\operatorname*{ind}_{Q^{(k+1)}}R^{(j+1)}$ to be
any element from the reduced index set $I^{\text{red}}.$

\item[Step 3:] We summarize and set%
\[
R_{i}(Q^{(k+1)})=R^{(j)}\quad\text{iff\quad}i=\operatorname*{ind}_{Q^{(k+1)}%
}R^{(j)}%
\]
and the index set%
\[
I(Q^{(k+1)})=\{\operatorname*{ind}_{Q^{(k+1)}}(R^{(j)}):R^{(j)}\in
\mathcal{O}_{m}(Q^{(k+1)})\}.
\]
It follows from the construction step 2 that the enumeration $R$ and the index
sets $I(Q^{(k)})$ have the desired property $(\ref{eq:zerl2}).$\qedhere
\end{description}
\end{proof}
For $1\leq i\leq M_1 q^m$ we recall the meaning of $\mathcal{A}_{m,i}%
\subseteq\mathcal{A,}$ which was defined in the previous proof, as%
\[
\mathcal{A}_{m,i}=\{A\in\mathcal{A}:\exists B\in\mathcal{A},\text{ such that
}(A,B)\in\mathcal{C}_{m,i}\}.
\]
Due to Proposition \ref{prop:boundednumber}, we can define an injective
mapping $\tau$ on $\mathcal{A}_{m,i}:$

\begin{definition}
\label{def:rearr}We define%
\begin{align*}
\tau & :\mathcal{A}_{m,i}\rightarrow\mathcal{A}\\
& A\mapsto\tau(A)
\end{align*}
through the relation%
\[
\tau(A)=B\quad\text{iff\quad}(A,B)\in\mathcal{C}_{m,i}.
\]
Additionally we get an inverse of $\tau$ on $\tau(\mathcal{A}_{m,i})$%
\[
\tau^{-1}(B)=A\quad\text{iff\quad}(A,B)\in\mathcal{C}_{m,i}.
\]
\end{definition}

\subsection{Decomposition of $\mathcal{C}_{m,i}$ using Arithmetic Progressions\label{sec:sep}}
\begin{proposition}
\label{prop:deccmi}For all $C>0$ there is a constant $M$ that depends only on $C$ and the space of homogeneous type $X$ such that we have the decomposition%
\[
\mathcal{C}_{m,i}=\mathcal{G}_{1}\cup\cdots\cup\mathcal{G}_{M},
\]
with the property that for all $1\leq l\leq M,$ $n\in
-\mathbb{N},$ and all disjoint $A_1,A_2$ in $\mathcal{A}_n$ with
\[
(A_{1},\tau(A_{1}))\in\mathcal{G}_{l}\quad\text{and\quad}(A_{2},\tau(A_{2}))\in
\mathcal{G}_{l}%
\]
the following separation of these sets holds:%
\begin{equation}
d(\tau^i(A_{1}),\tau^j(A_{2}))>Cq^{n}\quad\text{for all }i,j\in\{0,1\}.\label{eq:sep}%
\end{equation}
Here, $\tau^0(A):=A$ and $\tau^1(A):=\tau(A)$.
\end{proposition}

\begin{proof}
Let $\{(Q^{(k)},\tau(Q^{(k)})):k\in\mathbb{N\}}$ be an enumeration of $\mathcal{C}%
_{m,i}.$ Initialize the collections $\mathcal{G}_{1},\ldots,\mathcal{G}_{M}$ as
empty. For $k\in\mathbb{N}$, we inductively add $(Q^{(k)},\tau(Q^{(k)}))$ to $\mathcal{G}_{r}$ for%
\[
r:=\min\{i\in\mathbb{N}:\text{ for all }(A_1,\tau(A_1))\in\mathcal{G}_{i}\text{ we have
}(\ref{eq:sep}) \text{ with }A_2\text{ replaced by }Q^{(k)}\}.
\]
Thus for $(A,\tau(A))\in\mathcal{G}_{L+1}$ we have  $(A,\tau(A))\notin\mathcal{G}_{l}$ for all $l\leq L$
and so we have that for all $l\leq L$ there exists a pair $(A_{l}^{0},\tau(A_{l}^{0}))\in
\mathcal{G}_{l}$ such that one of the four expressions
\[
d(A,A_l^0),\quad d(A,\tau(A_{l}^{0})),\quad d(\tau(A),A_l^0),\quad d(\tau(A),\tau(A_l^0))
\]
is $\leq Cq^n$. According to the properties of the collection $\mathcal{C}_{m,i}$, the sets in the collection $\{A_l^0\}_{l=1}^L$ as well as the sets $\{\tau(A_l^0)\}_{l=1}^L$ are disjoint. So the remark after Lemma \ref{lem:max} yields that $L$ can't be greater than $4(M_0+1)$ with $M_0$ depending only on $C$ and on the space of homogeneous type $X$. This proves the proposition.
\end{proof}

We cannot guarantee that a dyadic $A$ cube divides into $N(A)\geq 2$ subcubes, but nevertheless we have as a consequence of the normality of $X$:
\begin{lemma}
\label{lem:punktmass}There exists a constant $L$ such that for every $l\geq L
$ we have that $A\in\mathcal{A}_{n}$, $B\in\mathcal{A}_{n-l}$ imply that $A\neq B$.
\end{lemma}


We now fix $\mathcal{G}=\mathcal{G}_{l}$ in the decomposition of Proposition
\ref{prop:deccmi} for some $l\leq M$ and introduce levels using arithmetic
progressions. We set%
\[
\mathcal{A}_{\mathcal{G}}:=\{A\in\mathcal{A}_{m,i}:(A,\tau(A))\in\mathcal{G}\}
\]
and the levels%
\begin{equation}
\mathcal{L}_{r}=\mathcal{A}_{\mathcal{G}}\cap%
{\displaystyle\bigcup\limits_{l=0}^{\infty}}
\mathcal{A}_{-l\cdot L(m+1)-r},\quad\text{where }0\leq r\leq
L(m+1)-1\label{eq:defL}%
\end{equation}
and $L\in\mathbb{N}$ is chosen in such a way that the condition of Lemma
\ref{lem:punktmass} is satisfied. We will later (in Section \ref{sec:step4}) give additional conditions on
the constant $L.$ Given a set $A\in\mathcal{L}_{r}$ we now define appropriate
predecessors. 
\begin{definition}
If $A\in\mathcal{A}_{-l\cdot L(m+1)-r},$ we define the \emph{arithmetic
predecessor}%
\begin{equation}
\widetilde{A}\label{eq:arithpre}
\end{equation}
to be the unique element in $\mathcal{A}_{-(l-1)\cdot L(m+1)-r},$ such that
$\widetilde{A}\supset A.$
\end{definition}
This works only if $l\geq1.$ If $l=0,$ we simply set
$\widetilde{A}:=X.$ We remark that for $A\in\mathcal{L}_{r}$ we have obviously
\[
\widetilde{A}\in{\displaystyle\bigcup\limits_{l=0}^{\infty}}
\mathcal{A}_{-l\cdot L(m+1)-r}
\quad\text{or}\quad \widetilde{A}=X,
\]
but not necessarily that $\widetilde{A}\in\mathcal{A}_{\mathcal{G}}$, hence  $\widetilde{A}$ need not be in $\mathcal{L}_{r}.$ \\
Note that the dyadic predecessor of $A$, denoted $\parent A$, defined in $(\ref{eq:defparent})$ does not coincide with the arithmetic predecessor $\widetilde{A}$ defined above.

\begin{definition}
Let $\mathcal{Z}$ be a collection of sets. $\mathcal{Z}$ is said to be \emph{nested}, if for all $A,B\in\mathcal{Z}$ we have that either
\[
A\cap B=\emptyset\quad\text{or}\quad A\subseteq{B}\quad\text{or}\quad B\subseteq A
\]
holds.
\end{definition}

The main result of this section is the following combinatorial theorem. It is the foundation of our work in the subsequent sections. It translates into norm estimates for rearrangement and shift operators in Section \ref{sec:rearr}. The significance of Theorem \ref{th:zerl} can be seen by examining the proof of T. Figiel \cite{Figiel1988}. To anticipate the notation used in the following theorem, we note that $\mathcal{H}$ will be the collection of cubes $A$ such that $\tau(A)$ has the same arithmetic predecessor as $A$. $\mathcal{I}$ will be the collection of cubes $A$ such that both $A$ and $\tau(A)$ are well inside their arithmetic predecessors and the collection $\mathcal{J}$ consists of the rest, where we again divide into the cases where either $A$ or $\tau(A)$ or both of them lie near the boundary of their arithmetic predecessors and call the corresponding collections $\mathcal{J}_1,\mathcal{J}_2$ and $\mathcal{J}_3$ respectively.
\begin{theorem}
\label{th:zerl}For $r\leq L(m+1)-1$ and $\mathcal{L}=\mathcal{L}%
_{r}$ defined by (\ref{eq:defL}), then for $\mathcal{L}$ there exists a decomposition%
\[
\mathcal{L}=\mathcal{H}\cup\mathcal{I}\cup\mathcal{J},
\]
such that

\begin{enumerate}
\item The collection
\[
\{A,\tau(A),A\cup\tau(A):A\in\mathcal{H}\}
\]
is nested.

\item $\mathcal{I}$ admits a decomposition as $\mathcal{I}=\mathcal{I}_{1}%
\cup\mathcal{I}_{2},$ so that the two collections%
\[
\{A,\tau(A),A\cup\tau(A):A\in\mathcal{I}_{j}\}\quad\text{for }j\in\{1,2\}
\]
are nested.

\item $\mathcal{J}$ admits a decomposition as $\mathcal{J}=\mathcal{J}_{1}%
\cup\mathcal{J}_{2}\cup\mathcal{J}_{3}$ such that we have

\begin{enumerate}
\item There exists an injection $\gamma_{1}%
:\mathcal{J}_{1}\cup\mathcal{J}_3\rightarrow\mathcal{A}$ such that the collection%
\[
\{A,\gamma_{1}(A),A\cup\gamma_{1}(A):A\in\mathcal{J}_{1}\}
\]
is nested and in addition we have for $A\in\mathcal{J}_{1}$%
\[
\gamma_{1}(A)\subseteq\widetilde{A},\text{\quad}d(\gamma_{1}(A),\com{\widetilde
{A}})\geq q^{\operatorname*{lev}A}\text{ \quad and\quad}d(\tau
(A),\com{\widetilde{\tau(A)}})\geq q^{\operatorname*{lev}A}.
\]

\item There exists an injection $\gamma_{2}%
:\mathcal{J}_{2}\cup\mathcal{J}_3\rightarrow\mathcal{A}$ such that the collection%
\[
\{\tau(A),\gamma_{2}(\tau(A)),\tau(A)\cup\gamma_{2}(\tau(A)):A\in
\mathcal{J}_{2}\}
\]
is nested and in addition we have for $A\in\mathcal{J}_{2}$%
\[
\gamma_{2}(\tau(A))\subseteq\widetilde{\tau(A)}\text{,\quad}d(\gamma_{2}%
(\tau(A)),\com{\widetilde{\tau(A)}})\geq q^{\operatorname*{lev}A}\quad
\text{and\quad}d(A,\com{\widetilde{A}})\geq q^{\operatorname*{lev}A}.
\]

\item For $\mathcal{J}_{3}$ and the injections $\gamma_{1}$ and $\gamma_{2} $
defined in (a) and (b), we have for $A\in\mathcal{J}_{3}$%
\[
d(\gamma_{1}(A),\com{\widetilde{A}})\geq q^{\operatorname*{lev}A}\quad
\text{and}\quad d(\gamma_{2}(\tau(A)),\com{\widetilde{\tau(A)}})\geq
q^{\operatorname*{lev}A}.
\]
Additionally, the two collections%
\[
\{A,\gamma_{1}(A),A\cup\gamma_{1}(A):A\in\mathcal{J}_{3}\}\quad\text{and\quad
}\{\tau(A),\gamma_{2}(\tau(A)),\tau(A)\cup\gamma_{2}(\tau(A)):A\in
\mathcal{J}_{3}\}
\]
are nested.
\end{enumerate}
\end{enumerate}
\end{theorem}
The proof of this theorem is divided into four basic steps.
\begin{description}
\item[Step 1 (Subsection \ref{sec:step1})] We give the definition of the decomposition of $\mathcal{L}$ into $\mathcal{H},\mathcal{I},\mathcal{J}$ and we further define the decomposition of $\mathcal{I}$ into $\mathcal{I}_1,\mathcal{I}_2$ and also the decomposition of $\mathcal{J}$ into $\mathcal{J}_1,\mathcal{J}_2,\mathcal{J}_3$.
\item[Step 2 (Subsection \ref{sec:step2})] We verify that $\mathcal{H}$ satisfies condition 1. of Theorem \ref{th:zerl}.
\item[Step 3 (Subsection \ref{sec:step3})] We verify that $\mathcal{I}_1,\mathcal{I}_2$ satisfy condition 2. of Theorem \ref{th:zerl}. This involves a two-coloring of $\mathcal{I}$ and an application of the argument in Step 2.
\item[Step 4 (Subsection \ref{sec:step4})] We first define the injections $\gamma_1,\gamma_2$ and verify condition 3. of the theorem. Here we use reduction to the arguments introduced in Steps 2. and 3. 
\end{description}

\subsubsection{Definition of the Decomposition of $\mathcal{L}$}\label{sec:step1}

Fix $A\in\mathcal{L.}$ We make the following case distinction:

\begin{enumerate}\vspace{-0.2cm}
\item If $\widetilde{A}=\widetilde{\tau(A)},$ we add $A$ to $\mathcal{H}.$\vspace{-0.2cm}

\item If $\widetilde{A}\cap\widetilde{\tau(A)}=\emptyset,$ we let\vspace{-0.2cm}

\begin{enumerate}
\item $A\in\mathcal{I},$ if the values of $d(A,\com{\widetilde{A}})$ and
$d(\tau(A),\com{\widetilde{\tau(A)}})$ are both greater or equal
$q^{\operatorname*{lev}A},$

\item $A\in\mathcal{J},$ if one of the values $d(A,\com{\widetilde{A}})$ or
$d(\tau(A),\com{\widetilde{\tau(A)}})$ is less than $q^{\operatorname*{lev}A}.$
\end{enumerate}
\end{enumerate}

For the case 2a we define the following collections: Take any $A\in
\mathcal{I}\cup\tau(\mathcal{I})$, define%
\[
\mathcal{P}(A):=\{B\in\mathcal{I}:\operatorname*{lev}B<\operatorname*{lev}%
A\text{,}\text{ }[(B\cap A\neq\emptyset\wedge\tau(B)\cap \com{A}\neq
\emptyset)\vee(B\cap \com{A}\neq\emptyset\wedge\tau(B)\cap A\neq\emptyset)]\}
\]
and set%
\[
\mathcal{R}(A):=\{J,\tau(J):J\in\mathcal{P}(A)\}.
\]
The purpose of the collection $\mathcal{P}(A)$ is that we get rid of
overlappings that occur if we define a two-coloring on $\mathcal{I}$ (say with
the colors black and white) and set%
\[
\mathcal{I}_{1}:=\{A\in\mathcal{I}:\operatorname*{color}A=\text{black}%
\},\quad\mathcal{I}_{2}:=\{A\in\mathcal{I}:\operatorname*{color}%
A=\text{white}\}.
\]
This two-coloring will have the crucial property that if $A\in\mathcal{I}$ is
white, then every element in $\mathcal{P}(A)$ is black. At last, we define a
decomposition of $\mathcal{J}$ and let%
\begin{align*}
\mathcal{J}_{1}  & :=\{A\in\mathcal{J}:d(A,\com{\widetilde{A}}%
)<q^{\operatorname*{lev}A}\text{ and }d(\tau(A),\com{\widetilde{\tau(A)}})\geq
q^{\operatorname*{lev}A}\}\\
\mathcal{J}_{2}  & :=\{A\in\mathcal{J}:d(A,\com{\widetilde{A}})\geq
q^{\operatorname*{lev}A}\text{ \ and }d(\tau(A),\com{\widetilde{\tau(A)}%
})<q^{\operatorname*{lev}A}\}\\
\mathcal{J}_{3}  & :=\mathcal{J}\backslash(\mathcal{J}_{1}\cup\mathcal{J}%
_{2}).
\end{align*}

\subsubsection{The Collection $\mathcal{H}$}\label{sec:step2}

We first analyse the collection $\mathcal{H},$ which is simpler to handle than
$\mathcal{I}$ and $\mathcal{J}.$

\begin{lemma}
\label{lem:colH}The collection $\{A,\tau(A),A\cup\tau(A):A\in\mathcal{H}\}$ is nested.
\end{lemma}

\begin{proof}
Let $A,B\in\mathcal{H}$ with $A\neq B\mathcal{.}$ It suffices to look at the
pairs $(A,B\cup\tau(B)),$ $(\tau(A),B\cup\tau(B)),$ $(A\cup\tau(A),B\cup
\tau(B)),$ since the other cases are trivial (this is the case if both
elements in the pair are dyadic cubes themselves) or considered by symmetry
(as for example the pair $(B,A\cup\tau(A))$). We begin with $(A,B\cup
\tau(B)):$\newline We assume%
\begin{equation}
A\cap(B\cup\tau(B))\neq\emptyset,\label{eq:menge}%
\end{equation}
Then we have to show that either $A\subseteq B\cup\tau(B)$ or $B\cup
\tau(B)\subseteq A.$ We have $(\ref{eq:menge})$ if and only if%
\[
A\cap B\neq\emptyset\text{ or }A\cap\tau(B)\neq\emptyset.
\]
For $A\cap B\neq\emptyset$ we have the three possibilities%
\[
A=B\quad\text{or\quad}A\subset B\quad\text{or\quad}B\subset A,
\]
where $\subset$ denotes a strict inclusion. Indeed these are the only cases
that can happen, since $A$ and $B$ are dyadic cubes. But $A=B$ is impossible,
since we assumed $A\neq B.$ If $A\subset B,$ we clearly have $A\subseteq
B\cup\tau(B).$ If $B\subset A,$ it holds also that%
\[
\widetilde{B}\subseteq A.
\]
This yields $\tau(B)\subseteq A,$ since $\widetilde{B}=\widetilde{\tau(B)}. $
So $B\cup\tau(B)\subseteq A.$ For the case $A\cap\tau(B)\neq\emptyset,$
analogous arguments complete the analysis of the pair $(A,B\cup\tau(B)). $

The pair $(\tau(A),B\cup\tau(B))$ is then treated in the same manner. \newline
We now come to $(A\cup\tau(A),B\cup\tau(B)):$ Again, we have to consider a few
cases. First we assume that
\[
(A\cup\tau(A))\cap(B\cup\tau(B))\neq\emptyset.
\]
This is the case if and only if
\[
A\cap B\neq\emptyset\quad\text{or\quad}A\cap\tau(B)\neq\emptyset
\quad\text{or\quad}\tau(A)\cap B\neq\emptyset\quad\text{or\quad}\tau
(A)\cap\tau(B)\neq\emptyset.
\]
These four cases are treated in the same way as above.
\end{proof}
\subsubsection{The Collection $\mathcal{I}$}\label{sec:step3}
\begin{lemma}
\label{lem:klgl1}For each $B\in\mathcal{I}\cup\tau(\mathcal{I})$ there exists
at most one $A\in\mathcal{I}\cup\tau(\mathcal{I})$ such that%
\[
B\in\mathcal{R}(A).
\]

\end{lemma}

\begin{proof}
Let $A_{1},A_{2}\in\mathcal{I}\cup\tau(\mathcal{I})$ with $A_{1}\neq A_{2} $
such that%
\[
B\in\mathcal{R}(A_{1})\quad\text{and\quad}B\in\mathcal{R}(A_{2}).
\]
We split the proof into two parts. Part (a) treats the case
$\operatorname*{lev}A_{1}=\operatorname*{lev}A_{2}$ and part (b) treats the
case $\operatorname*{lev}A_{1}<\operatorname*{lev}A_{2},$ which is the general
case, since we can always exchange $A_{1}$ and $A_{2}.$ We additionally assume
$B\in\mathcal{I},$ since the argument is symmetric if we assume $B\in
\tau(\mathcal{I}).$

\begin{enumerate}\vspace{-0.2cm}
\item[(a)] We first treat the case $\operatorname*{lev}A_{1}%
=\operatorname*{lev}A_{2}.$ Here we get from the definition of $\mathcal{I}$
and from Proposition \ref{prop:deccmi} that $d(A_{1},A_{2}%
)>q^{\operatorname*{lev}A_{1}}$ and $\operatorname*{lev}B\leq
\operatorname*{lev}A_{1}-L(m+1).$ Again we distinguish two cases. In view of
the fact that $B\in\mathcal{R}(A_{1})$, we split to (i) $B\cap A_{1}%
\neq\emptyset$ and (ii) $\tau(B)\cap A_{1}\neq\emptyset.$\vspace{-0.2cm}
\begin{enumerate}
\item[(i)] With $B\cap A_{1}\neq\emptyset$ it holds that $B\subset A_{1}$ and
so $B\cap A_{2}=\emptyset.$ Thus we have%
\[
d(A_{1},\tau(B))\leq d(B,\tau(B))\leq q^{\operatorname*{lev}B+m}\leq
q^{\operatorname*{lev}A_{1}}.
\]
From these facts we infer that $\tau(B)\not \subset A_{2}$ and that implies
$\tau(B)\cap A_{2}=\emptyset,$ which contradicts the assumption $B\in
\mathcal{R}(A_{2}).$
\item[(ii)] If $\tau(B)\cap A_{1}\neq\emptyset,$ this leads to $\tau(B)\subset
A_{1}$ and $\tau(B)\cap A_{2}=\emptyset.$ Analogously to the above we get%
\[
d(A_{1},B)\leq d(\tau(B),B)\leq q^{\operatorname*{lev}B+m}\leq
q^{\operatorname*{lev}A_{1}}.
\]
This implies $B\cap A_{2}=\emptyset,$ which contradicts $B\in\mathcal{R}%
(A_{2}).$
\end{enumerate}\vspace{-0.2cm}

\item[(b)] Now we assume without loss of generality that $\operatorname*{lev}%
A_{1}<\operatorname*{lev}A_{2}.$ Here we consider the two cases (i)
$A_{1}\subset A_{2}$ and (ii) $A_{1}\cap A_{2}=\emptyset.$
\vspace{-0.2cm}
\begin{enumerate}
\item[(i)] For $A_{1}\subset A_{2}$ we have by definition of $\mathcal{I}$%
\begin{equation}
d(A_{1},\com{A_{2}})\geq d(A_{1},\com{\widetilde{A_{1}}})>q^{\operatorname*{lev}%
A_{1}}.\label{eq:lem3}%
\end{equation}
Like in case (a) we have to consider the two cases $B\cap A_{1}\neq
\emptyset\ $and $\tau(B)\cap A_{1}\neq\emptyset.$ We proceed with $B\cap
A_{1}\neq\emptyset.$ (The case $\tau(B)\cap A_{1}\neq\emptyset$ works
analogously.) So it follows that $B\subset A_{1}$ and so $B\subset A_{2}.$ We
have the estimate%
\begin{equation}
d(A_{1},\tau(B))\leq d(B,\tau(B))\leq q^{m+\operatorname*{lev}B}\leq
q^{\operatorname*{lev}A_{1}}.\label{eq:lem4}%
\end{equation}
Now it follows from $(\ref{eq:lem3})$ and $(\ref{eq:lem4})$ that
$\tau(B)\not \subset \com{A_{2}},$ i.e. $\tau(B)\subset A_{2}.$ This contradicts
$B\in\mathcal{R}(A_{2}).$

\item[(ii)] Let $A_{1}\cap A_{2}=\emptyset.$ In that case we have%
\[
d(A_{1},A_{2})\geq d(A_{1},\com{\widetilde{A_{1}}})>q^{\operatorname*{lev}A_{1}%
}.
\]
If $B\cap A_{1}\neq\emptyset$ (the other case $\tau(B)\cap A_{1}\neq\emptyset$
is treated analogously), it follows that $B\subset A_{1}$ and hence $B\subset
\com{A_{2}}.$ We have%
\[
d(A_{1},\tau(B))\leq d(B,\tau(B))\leq q^{m+\operatorname*{lev}B}\leq
q^{\operatorname*{lev}A_{1}}.
\]
Thus we get%
\[
\tau(B)\not \subset A_{2}%
\]
and thus $\tau(B)\cap A_{2}=\emptyset.$ This identity together with $B\subset
\com{A_{2}}$ contradicts $B\in\mathcal{R}(A_{2}).$ This finishes the proof.
\end{enumerate}
\end{enumerate}
\vspace{-0.8cm}
\end{proof}
\vspace{0.3cm}

Lemma \ref{lem:klgl1} allows us to introduce the announced two-coloring on
$\mathcal{I}\cup\tau(\mathcal{I})$ with the colors black and white that satisfies the following three conditions:

\begin{enumerate}\vspace{-0.2cm}
\item For each $A\in\mathcal{I}\cup\tau(\mathcal{I})$ the collection
$\mathcal{R}(A)$ is monochromatic,\vspace{-0.2cm}

\item If the color of $A\in\mathcal{I}\cup\tau(\mathcal{I})$ is already
determined, then each $B\in\mathcal{R}(A)$ satisfies%
\[
\text{color}(B)\neq\text{color}(A),\vspace{-0.2cm}
\]

\item For each $A\in\mathcal{I},$%
\[
\text{color}(A)=\text{color}(\tau(A)).
\]\vspace{-0.2cm}
\end{enumerate}

Define%
\[
\mathcal{I}_{1}=\{A\in\mathcal{I}:\text{color}(A)=\text{white}\}\quad
\text{and\quad}\mathcal{I}_{2}=\{A\in\mathcal{I}:\text{color}(A)=\text{black}%
\}.
\]

\begin{lemma}
\label{lem:nichtR}If $A\in\mathcal{I}$ and $B\notin\mathcal{P}(A)$ with
$\operatorname*{lev}B<\operatorname*{lev}A,$ then%
\[
B\cup\tau(B)\subseteq A\quad\text{or\quad}B\cup\tau(B)\subseteq \com{A}.
\]

\end{lemma}

\begin{proof}
This is nothing else but a logical manipulation of the definition of
$\mathcal{P}(A).$
\end{proof}

\begin{lemma}
\label{lem:schoenI}The two subcollections
\[
\{A,\tau(A),A\cup\tau(A):A\in\mathcal{I}_{j}\}\quad\text{for }j\in\{1,2\}
\]
are nested.
\end{lemma}

\begin{proof}
Let $A,B\in\mathcal{I}_{j}$ for $j\in\{1,2\}.$ We consider the three pairs (a)
$(A,B\cup\tau(B)),$ (b) $(\tau(A),B\cup\tau(B)),$ (c) $(A\cup\tau(A),B\cup
\tau(B)).$\vspace{-0.2cm}
\begin{enumerate}
\item[(a)] We have to show that either%
\begin{equation}
A\cap(B\cup\tau(B))=\emptyset\quad\text{or\quad}A\subseteq B\cup\tau
(B)\quad\text{or\quad}B\cup\tau(B)\subseteq A.\label{eq:lemnest}%
\end{equation}
We consider the three cases (i) $\operatorname*{lev}A=\operatorname*{lev}B,$
(ii) $\operatorname*{lev}A<\operatorname*{lev}B$ and (iii)
$\operatorname*{lev}B<\operatorname*{lev}A:$
\begin{enumerate}
\item[(i)] This is clear, since $A$ and $B$ are dyadic cubes.
\item[(ii)] If $A\in\mathcal{P}(B)$ then either $A$ or $B$ is not in
$\mathcal{I}_{j};$ if $A\notin\mathcal{P}(B),$ then due to Lemma
\ref{lem:nichtR} we have%
\[
A\cup\tau(A)\subseteq B\quad\text{or\quad}A\cup\tau(A)\subseteq \com{B}.
\]
In the first case, clearly, $A\subseteq B\cup\tau(B).$ In the second case
$A\subseteq \com{B}$ and so%
\[
A\cap(B\cup\tau(B))=A\cap\tau(B)=\left\{
\begin{array}
[c]{ll}%
A, & \text{if }A\cap\tau(B)\neq\emptyset\\
\emptyset, & \text{else}%
\end{array}
.\right.
\]
Both branches lead to one of the alternatives in $(\ref{eq:lemnest}).$
\item[(iii)] Analogous to (even simpler than) case (ii).
\end{enumerate}\vspace{-0.2cm}
\item[(b)] Analogous to (a).\vspace{-0.2cm}
\item[(c)] We have to show that either%
\begin{equation}
(A\cup\tau(A))\cap(B\cup\tau(B))=\emptyset\quad\text{or\quad}A\cup
\tau(A)\subseteq B\cup\tau(B)\quad\text{or\quad}B\cup\tau(B)\subseteq
A\cup\tau(A).
\end{equation}
We consider the two cases (i) $\operatorname*{lev}A=\operatorname*{lev}B,$
(ii) $\operatorname*{lev}A<\operatorname*{lev}B$:
\begin{enumerate}\vspace{-0.2cm}
\item[(i)] Since $A$ and $B$ are in a collection $\mathcal{G},$ we have that
$d(\tau(A),B)$ and $d(A,\tau(B))$ are greater than $q^{\operatorname*{lev}A},$
and so%
\[
\left(  A\cup\tau(A)\right)  \cap\left(  B\cup\tau(B)\right)  =\emptyset.
\]
\vspace{-0.2cm}
\item[(ii)] If $A\in\mathcal{P}(B)$ then either $A$ or $B$ is not in
$\mathcal{I}_{j}.$ If $A\notin\mathcal{P}(B)$ we get with Lemma
\ref{lem:nichtR} that either%
\[
A\cup\tau(A)\subseteq B\quad\text{or\quad}A\cup\tau(A)\subseteq \com{B}.
\]
In the first case, clearly, $A\cup\tau(A)\subseteq B\cup\tau(B).$ In the
second case we get from part (b) of the Lemma that for
\[
(A\cup\tau(A))\cap\tau(B)
\]
we only have the three possibilities $\emptyset,$ $A\cup\tau(A)$ or $\tau(B).$
The former two lead to $(A\cup\tau(A))\cap(B\cup\tau(B))=\emptyset$ and
$A\cup\tau(A)\subseteq B\cup\tau(B)$ respectively. The third one gives%
\[
\tau(B)\subseteq A\cup\tau(A),
\]
which is not possible (cf. Lemma \ref{lem:punktmass}).
\end{enumerate}
\end{enumerate}
\vspace{-0.7cm}
\end{proof}

\begin{remark}
We remark that this decomposition of $\mathcal{I}$ into $\mathcal{I}_1$ and $\mathcal{I}_2$, in particular the proof of Lemma \ref{lem:klgl1}, does not depend on the explicit form of $\mathcal{I}$ and, what is even more important, the corresponding injection $\tau$. In fact, the same proof works if there exists a constant $C_R$ such that the following conditions are satisfied:
\begin{enumerate}
\item $\tau$ is an injection on $\mathcal{I}$ that $\lev A=\lev \tau(A)$ whenever $A\in\mathcal{I}$,
\item for every $A\in\mathcal{I}, \tilde{A}\cap\widetilde{\tau(A)}=\emptyset,$
\item for $A\in\mathcal{I}$, we have that $\min(d(A,\complement\tilde{A}),d(\tau(A),\complement\widetilde{\tau(A)}))>C_R q^{\lev A}$,
\item for $A,B\in\mathcal{I}$, it holds that $d(\tau^j(A),\tau^i(B))>C_R q^{\lev A}$ for $i,j\in \{0,1\}$,
\item for $A\in\mathcal{I}$, we have that $\max(d(\tilde{A},\tau(A)),d(\widetilde{\tau(A)},A))\leq C_R q^{\lev\tilde{A}}$,
\item for two disjoint sets $A,B$ in $\mathcal{I}$ such that $A\supset B$, we have that $\lev A\geq \lev B+L(m+1)$.

\end{enumerate}
\end{remark}
\subsubsection{The Collection $\mathcal{J}$}\label{sec:step4}

\begin{lemma}
There exists a constant $C_2$ such that for all $A\in\mathcal{A}_{n}$ and every $l\in\mathbb{N,}$ the number $Y_l^A$ of sets $B$ in $\mathcal{A}%
_{n-l}$, for which we have $B\subseteq A$, is bounded from
below by%
\[
C_{2}q^{l}.
\]
\end{lemma}

\begin{proof}
If we use the normality of $X$ and point 4. of Theorem \ref{th:dyadic} the conclusion of the lemma follows from the subsequent chain of inequalities:
\begin{eqnarray*}
\frac{b_1c_1}{b_2c_2}q^n&\leq& \frac{1}{b_2c_2}\mu(A)=\frac{1}{b_2c_2}\sum_{B\subset A, B\in\mathcal{A}_{n-l}}\mu(B) \\
&\leq & \sum_{B\subset A, B\in\mathcal{A}_{n-l}} q^{n-l}=Y^A_l  q^{n-l}.\qedhere 
\end{eqnarray*}


\end{proof}

Now, recall the definition of the boundary layer $\partial_{t}A$ of a cube $A$ with level $n$, which we defined as
\[
\partial_t A=\{x\in A:d(x,X\backslash A)\leq t q^n\}.
\]
Additionally, due to Theorem \ref{th:dyadic}, the measure of $\partial_t A$ admits the following upper bound
\[
\mu(\partial_t A)<c_3 t^\eta\mu(A)
\]
for some universal constants $c_3,\eta>0$.
\begin{lemma}
There exists a constant $C_3$ such that for all $A\in\mathcal{A}_n$ and every $l\in\mathbb{N}$, the number $X_l^A$ of sets $B\in\mathcal{A}_{n-l}$ for which we have 
\[
B\cap \partial _{q^{-l}}A\neq \emptyset
\]
is bounded from above by
\[
C_3q^{l(1-\eta)}.
\]
\end{lemma}

\begin{proof}
It is a simple consequence of the quasi-triangle inequality that there exists $d\geq 1$ depending only on $X$ such that if $B\in\mathcal{A}_{n-l}$ we have
\[
B\cap\partial_{q^{-l}}\neq \emptyset\Rightarrow B\subset \partial_{dq^{-l}}A.
\]
With this fact, the normality of $X$ and Theorem \ref{th:dyadic}, point 4. and 5., the conclusion of the lemma follows from the subsequent chain of inequalities:
\begin{eqnarray*}
X^A_l  q^{n-l}&\leq& b_2c_2 \sum_{B\subset\partial_{dq^{-l}}A,B\in\mathcal{A}_{n-l}} q^{n-l}\leq \frac{1}{b_1c_1} \sum_{B\subset\partial_{dq^{-l}}A,B\in\mathcal{A}_{n-l}} \mu(B) \\
&\leq& \frac{1}{b_1c_1}\mu(\partial_{dq^{-l}}A)\leq c_3\frac{1}{b_1c_1} d^{\eta}q^{-l\eta}\mu(A)\leq c_3\frac{b_2c_2}{b_1c_1} d^{\eta}q^{-l\eta}q^n.\qedhere
\end{eqnarray*}
\end{proof}

In view of the above two lemmas and Lemma \ref{lem:punktmass}, we can choose
$L$ in $(\ref{eq:defL})$ large enough that for all $l\geq L,$ $A\in
\mathcal{A}_{n}$ and $B\in\mathcal{A}_{n-l}$ we don't have%
\[
A=B,
\]
and in addition that the quotient $\frac{Y_{L}^{A}}{X_{L}^{A}}$ admits the
bound%
\[
\frac{Y_{L}^{A}}{X_{L}^{A}}>2.
\]
This property is crucial, since it enables us to define an injection
$\gamma_{1}:\mathcal{J}_{1}\cup\mathcal{J}_{3}\rightarrow\mathcal{A},$ such
that we have $\gamma_{1}(A)\subseteq\widetilde{A}$ and we have moved away from
the boundary of $\widetilde{A}:$%
\[
d(\gamma_{1}(A),\com{\widetilde{A}})\geq q^{\operatorname*{lev}A}.
\]
We extend $\gamma_{1}$ to $\mathcal{J}_{1}\cup\mathcal{J}_{3}\cup\gamma
_{1}(\mathcal{J}_{1})\cup\gamma_{1}(\mathcal{J}_{3})$ and define for
$A\in\gamma_{1}(\mathcal{J}_{1})\cup\gamma_{1}(\mathcal{J}_{3})$ that%
\[
\gamma_{1}(A):=\gamma_{1}^{-1}(A).
\]
It as now a straightforward consequence of the definitions that the following holds:

\begin{lemma}
\label{lem:colJ1}The collections%
\[
\{A,\gamma_{1}(A),A\cup\gamma_{1}(A):A\in\mathcal{J}_{1}\}\quad\text{and\quad
}\{A,\gamma_{1}(A),A\cup\gamma_{1}(A):A\in\mathcal{J}_{3}\}
\]
are nested.
\end{lemma}

\begin{proof}
The proof is in fact nothing else than the proof of Lemma \ref{lem:colH} with
$\tau$ replaced by $\gamma_{1}.$
\end{proof}

Analogously, we define an injection $\gamma_{2}:\tau(\mathcal{J}_{2})\cup
\tau(\mathcal{J}_{3})\rightarrow\mathcal{A},$ such that we have $\gamma
_{2}(\tau(A))\subseteq\widetilde{\tau(A)}$ and%
\[
d(\gamma_{2}(\tau(A)),\com{\widetilde{\tau(A)}})\geq q^{\operatorname*{lev}A}%
\]
and extend it to $\tau(\mathcal{J}_{2})\cup\tau(\mathcal{J}_{3})\cup\gamma
_{2}(\tau(\mathcal{J}_{2}))\cup\gamma_{2}(\tau(\mathcal{J}_{3}))$ by defining
for $A\in\gamma_{2}(\tau(\mathcal{J}_{2}))\cup\gamma_{2}(\tau(\mathcal{J}%
_{3})):$%
\[
\gamma_{2}(A):=\gamma_{2}^{-1}(A).
\]
Again it follows as in Lemma \ref{lem:colJ1} that

\begin{lemma}
\label{lem:colJ2}The collections%
\[
\{\tau(A),\gamma_{2}(\tau(A)),\tau(A)\cup\gamma_{2}(\tau(A)):A\in
\mathcal{J}_{2}\}\quad\text{and\quad}\{\tau(A),\gamma_{2}(\tau(A)),\tau
(A)\cup\gamma_{2}(\tau(A)):A\in\mathcal{J}_{3}\}
\]
are nested.
\end{lemma}
We can now summarize our considerations and thus prove our main theorem in
this section

\begin{proof}
[Proof of Theorem \ref{th:zerl}]The collections $\{A,\tau(A),A\cup
\tau(A):A\in\mathcal{H}\}$ and $\{A,\tau(A),A\cup\tau(A):A\in\mathcal{I}%
_{i}\}$ for $i\in\{1,2\}$ are nested by the Lemmas \ref{lem:colH}
and \ref{lem:schoenI} respectively. Lemmas \ref{lem:colJ1} and \ref{lem:colJ2}
yield that the collections $\{A,\gamma_{1}(A),A\cup\gamma_{1}(A):A\in
\mathcal{J}_{1}\},\{A,\gamma_{1}(A),A\cup\gamma_{1}(A):A\in\mathcal{J}%
_{3}\},\{\tau(A),\gamma_{2}(\tau(A)),\tau(A)\cup\gamma_{2}(\tau(A)):A\in
\mathcal{J}_{2}\}$ and $\{\tau(A),\gamma_{2}(\tau(A)),\tau(A)\cup\gamma
_{2}(\tau(A)):A\in\mathcal{J}_{3}\}$ are nested. The additional properties of
the mappings $\gamma_{1}$ and $\gamma_{2}$ follow from the definition. We have
thus completely proved the theorem.
\end{proof}

\section{Decomposing Singular Integral Operators}\label{sec:main}
In this section we decompose singular integral operators as absolutely convergent series of simple rearrangements, shifts and two paraproducts.
\subsection{Integral Operators}

We now define the integral operator $K$ with the kernel $k:X\times
X\rightarrow\mathbb{R,}$ $k\in L^{2}(X\times X)$ by%
\[
K(f)(x):=\int_{X}k(x,y)f(y)~d\mu(y)
\]
for $f\in L^{2}_E(X)$ and $E$ is a $\UMD$ Banach space. We assume structural estimates on $k$, in particular a strong off-diagonal decay and also a weak boundedness estimate on the diagonal. This is formalized with the following definition using the orthonormal basis from Lemma \ref{lem:basis}.
First recall that $q$ was the
number with that $q^{\lev A}$ represents roughly the "size" of $A. $
\begin{definition}\label{def:struct}
Let $k\in L^2(X\times X)$. We say that $k$ is an \emph{admissible kernel} if there exist $C_{S}>0$ and $\delta>0$ such that $|\left<k,1_X\otimes 1_X\right>|\leq C_S$ and for all $Q,R\in\mathcal{E}(\mathcal{A})$ with $\lev Q=\lev R$ we have
\begin{equation}
\left\vert \left\langle k,d_{Q}^{(\varepsilon_1)}\otimes d_{R}^{(\varepsilon_2)}\right\rangle \right\vert\leq
C_{S}\left( 1+\frac{d(\parent Q,\parent R)}{q^{\lev Q+1}}\right)  ^{-1-\delta},\quad\varepsilon=(\varepsilon_1,\varepsilon_2)\in\NS\label{eq:struct1}.
\end{equation}
\end{definition}

In this section we provide vector valued norm estimates for integral operators defined by admissible kernels. We point out that the $L^p$-norm of the integral operators depends just on the structural constants $C_S$ and $\delta$, the value of $p$ and the $\BMO$-norms of $K(1),K^*(1)$. In particular, the $L^2$-norm of $k$ is not present in the estimates. From now on, we work with admissible kernels $k$.
We expand the kernel $k$ in the isotropic orthonormal basis introduced in Section \ref{sec:isotr}. The division of $Z$ into three groups (see (\ref{eq:orthz})) gives rise to the following decomposition of the kernel $k$. 
We let
\begin{eqnarray*}
k_1&:=&\sum_{n=0}^{\infty}%
\sum_{A,B\in\mathcal{A}_{-n}}\sum_{Q\in\mathcal{E}(A)}\sum_{R\in
\mathcal{E}(B)}\left\langle k,d_{Q}\otimes d_{R}\right\rangle
d_{Q}\otimes d_{R}, \\
k_2&:=&\sum_{n=0}^{\infty}\sum_{A,B\in\mathcal{A}_{-n}}\sum_{Q\in
\mathcal{E}(A)}\frac{\left\langle k,d_{Q}\otimes1_{B}\right\rangle }{\mu(B)}d_{Q}\otimes 1_{B},\\
k_3&:=&\sum_{n=0}^{\infty}\sum_{A,B\in\mathcal{A}_{-n}}\sum_{R\in
\mathcal{E}(B)}\frac{\left\langle k,1_{A}\otimes d_{R}\right\rangle }%
{\mu(A)}1_{A}\otimes d_{R}.
\end{eqnarray*}
If we decompose $k$ into the isotropic orthonormal basis we see that
\[
k = \int_{X}\int_X k(s,t)d\mu(s)d\mu(t)+k_1+k_2+k_3.
\]
We note the following identities (which follow from $X=%
{\displaystyle\bigcup\limits_{A\in\mathcal{A}_{-n}}}
A$ for every $n\in\mathbb{N}_0)$%
\[
\sum_{B\in\mathcal{A}_{-n}}\left\langle k,d_{Q}\otimes1_{B}\right\rangle
=\left\langle K(1),d_{Q}\right\rangle ,\quad\sum_{A\in\mathcal{A}_{-n}%
}\left\langle k,1_{A}\otimes d_{R}\right\rangle =\left\langle K^{\ast
}(1),d_{R}\right\rangle.
\]
Now we let $f\in L^p_E(X)$ and $g\in L^q_{E'}(X)$ be finite linear combinations of Haar functions and $E$ be a $\UMD$-space. Then we see that $k_2$ has the further decomposition
\begin{equation}
\left<k_{2},g\otimes f\right>=B_2(f,g)+\widetilde{B_2}(f,g),
\end{equation}
where
\begin{eqnarray}
B_2(f,g)&:=&\sum_{n=0}^{\infty}\sum_{A,B\in\mathcal{A}_{-n}}\sum
_{Q\in\mathcal{E}(A)}\frac{\left\langle k,d_{Q}\otimes1_{B}\right\rangle }%
{\mu(B)}\left<d_{Q}\otimes\left(  1_{B}-\frac{\mu(B)}{\mu(A)}1_{A}\right),g\otimes f\right>\quad\text{and}\\
\widetilde{B_2}(f,g) &:=& \sum_{n=0}^{\infty}\sum_{A\in\mathcal{A}_{-n}}\sum_{Q\in\mathcal{E}%
(A)}\frac{\left\langle K(1),d_{Q}\right\rangle }{\mu(A)}\left<d_{Q}%
\otimes 1_{A},g\otimes f\right>.\label{eq:defb1b2}
\end{eqnarray}
We also decompose $k_3$ further and get the following identity, which is valid in $L^2_E(X)$
\begin{equation}
 \int_X k_3(x,y)f(y)d\mu(y) = K_3f(x)+\widetilde{K_3}f(x), 
\end{equation}
where
\begin{eqnarray} 
  K_3f(x)&:=&\sum_{n=0}^{\infty}\sum_{A,B\in\mathcal{A}_{-n}}\sum
_{R\in\mathcal{E}(B)}\frac{\left\langle k,1_{A}\otimes d_{R}\right\rangle
}{\mu(A)}\left<d_{R},f\right>\left(  1_{A}(x)-\frac{\mu(A)}{\mu(B)}1_{B}(x)\right)\quad\text{and}
\\
  \widetilde{K_3}f(x)&:=& \sum_{n=0}^{\infty}\sum_{B\in\mathcal{A}_{-n}}\sum_{R\in\mathcal{E}%
(B)}\frac{\left\langle K^{\ast}(1),d_{R}\right\rangle }{\mu(B)}%
\left<d_{R},f\right>1_{B}(x)\label{eq:defk2k3}
\end{eqnarray}
Furthermore we set 
\[
K_{1}f(x):=\int_{X}k_{1}(x,y)f(y)~d\mu(y).
\]

\subsection{Statement of the Main Theorems}
Recall the definition of the $\sigma$-algebras $\mathcal{F}^{\lev}_k,$ which were defined to be the $\sigma$-algebras generated by the dyadic cubes of level $-k$. In this section (Section \ref{sec:main}), each occurrence of $\BMO$ means the space $\BMO(X,\mathcal{F}^{\lev}_k)$ with these $\sigma$-algebras. Further, we let $E$ be a $\UMD$-space (see Section \ref{sec:SteinBourgain}). We now state the main result in this article.
\begin{theorem}
\label{th:main}Let $K$ be the integral operator defined in the last section satisfying (\ref{eq:struct1}).
Then $K$, initially defined on finite linear combinations of Haar functions, extends linearly to a unique bounded operator on $L^{p}$ for $1<p<\infty$, i.e. we have a constant
$C_{K}$ such that%
\[
\left\Vert K:L_{E}^{p}(X)\rightarrow L_{E}^{p}(X)\right\Vert \leq C_{K}%
\]
and $C_{K}$ depends only on $p,$ the $\BMO$--norms of $K(1)$ and
$K^{\ast}(1)$, the constants $C_{S}$ and $\delta$ coming from the structural estimate (\ref{eq:struct1}) and the $\UMD$ constant of $E$.
\end{theorem}
The starting point and basic idea of the proof is the following decomposition of the bilinear form $\left<Kf,g\right>$:
\begin{equation}\label{eq:decomp}
\left<Kf,g\right>=\widetilde{B_2}(f,g)+\left<\widetilde{K_3}f,g\right>+\sum_{m=0}^\infty \left<k_{1,m},g\otimes f\right>+B_{2,m}(f,g)+\left<K_{3,m}f,g\right>,
\end{equation}
where these operators are defined in (\ref{eq:defb1b2}),(\ref{eq:defk2k3}),(\ref{eq:defk1m}),(\ref{eq:defb1m}) and (\ref{eq:defk2m}). Clearly, we assumed here that $k$ has mean zero with respect to the product measure $\mu\otimes\mu$.
In fact, as we will see in the proof of Theorem \ref{th:k1} and the proof of Theorem \ref{th:k3}, this decomposition can be further split as
\begin{equation}
K=P_{K(1)}^*+P_{K^*(1)}+\sum_{m=0}^\infty\sum_{i=1}^{M_1 q^m}\left(\sum_{j,k=1}^{N-1} T_{m,i}^{(j,k)}\circ\mathcal{M}_{m,i}^{(j,k)}+\sum_{j=1}^{N-1}W_{m,i}^{(j)}\circ \widetilde{\mathcal{M}}_{m,i}^{(j)}+\sum_{k=1}^{N-1}U_{m,i}^{(k)}\circ \mathcal{M}_{m,i}^{(k)}\right),\label{eq:decomp2}
\end{equation}
where $P_{K(1)}$ and $P_{K^*(1)}$ are paraproducts defined in the proof of Theorem \ref{th:k3}, $T_{m,i}^{(j,k)},W_{m,i}^{(j)},U_{m,i}^{(k)}$ are shift and rearrangement operators defined in Section \ref{sec:rearr} and the operators $\mathcal{M}$ are suitable Haar multipliers. The five summands in (\ref{eq:decomp}) correspond to the summands in (\ref{eq:decomp2}) in the same order. 

\begin{remark}
We note explicitly that the constant $C_K$ in the last theorem does \emph{not} depend on the $L^2$-norm of $k(x,y)$, which is the crucial fact about the statement. It thus can be shown that 
\begin{enumerate}
\item Theorem \ref{th:main} yields a direct generalization of T. Figiel's $T(1)$ theorem (\cite{Figiel1990}) to spaces of homogeneous type and
\item Theorem \ref{th:main} yields a direct generalization of Coifmans $T(1)$ theorem (as presented in \cite{Christ1990book}, for the origin of the method see also \cite{CoifmanJonesSemmes1989}) to vector valued singular integral operators given by standard kernels. 
\end{enumerate}

\end{remark}

According to the decomposition of $\mathcal{C}$ we split 
$k_{1},$ $B_{1}$ and $K_{2}$ further and define
\begin{align}
k_{1,m}:=  & \sum_{(A,B)\in\mathcal{C}_{m}}\sum_{Q\in\mathcal{E}(A)}\sum
_{R\in\mathcal{E}(B)}\left\langle k,d_{Q}\otimes d_{R}\right\rangle
d_{Q}\otimes d_{R},\quad\text{in }L^2(X\times X)\label{eq:defk1m},\\
B_{2,m}(f,g):=  & \sum_{(A,B)\in\mathcal{C}_{m}}\sum_{Q\in\mathcal{E}(A)}%
\frac{\left\langle k,d_{Q}\otimes1_{B}\right\rangle }{\mu(B)}%
\left<d_{Q}\otimes\left(  1_{B}-\frac{\mu(B)}{\mu(A)}1_{A}\right),g\otimes f\right> \label{eq:defb1m} ,\\
K_{3,m}f:=  & \sum_{(A,B)\in\mathcal{C}_{m}}\sum_{R\in\mathcal{E}(B)}%
\frac{\left\langle k,1_{A}\otimes d_{R}\right\rangle }{\mu(A)}%
\left<d_{R},f\right>\left(  1_{A}-\frac{\mu(A)}{\mu(B)}1_{B}\right),\quad\text{in }L^2_E(X).\label{eq:defk2m}
\end{align}
Associated to the kernel $k_{1,m}$ we define the integral operator
\[
K_{1,m}(f)(x):=\int_{X}k_{1,m}(x,y)f(y)~d\mu(y).
\]
In later sections we prove the following theorems, from which our main
result (Theorem \ref{th:main}) follows.\newcounter{mycount} \setcounter{mycount}{\value{theorem}}
In the subsequent theorem, $\delta$ is the positive number coming from the structural estimate (\ref{eq:struct1}) and $q$ is the constant appearing in Theorem \ref{th:dyadic}.
\begin{theorem}
\label{th:k1}For all $1<p<\infty$ there exists a constant $C_{p}$ depending only on $p,X$, the $\UMD$ constant of $E$ and $C_S$ from (\ref{eq:struct1}), such that for all $f\in L^{p}_E(X),g\in L^{p'}_{E'}(X)$, which are finite linear combinations of Haar functions, the operators $K_{1,m},K_{3,m}$ and the bilinear form $B_{2,m}$ satisfy the following estimates:
\begin{align}
\left\Vert K_{1,m}(f)\right\Vert _{L^{p}_E(X)} & \leq C_{p}(m+1)q^{-m\delta}\left\Vert
f\right\Vert _{L^{p}_E(X)}, \label{eq:K1}\\
\left\Vert K_{3,m}(f)\right\Vert _{L^{p}_E(X)} & \leq C_{p}(m+1)q^{-m\delta}\left\Vert
f\right\Vert _{L^{p}_E(X)}, \label{eq:K2}\\
\left|B_{2,m}(f,g)\right| & \leq C_{p}(m+1)q^{-m\delta}\left\Vert
f\right\Vert _{L^{p}_E(X)}\left\|g\right\|_{L^{p'}_{E'}(X)}.\label{eq:K4}
\end{align}
Here, $p'=p/(p-1)$ denotes the conjugate exponent to $p$.
\end{theorem}

\begin{remark}
For this theorem, we need the $L^p$-boundedness of rearrangement and shift operators, which will be introduced in Section \ref{sec:rearr} and the boundedness of these operators will be proved in Sections \ref{sec:figiel} and \ref{sec:boundwithoutfigiel}.
\end{remark}

\begin{theorem}
\label{th:k3}For all $1<p<\infty$ there exists a constant $C_{p},$ which depends only on $p,X$ and the $\UMD$ constant of $E$ such that for all $f\in L^{p}_E(X),g\in L^{p'}_{E'}(X)$ which are finite linear combinations of Haar functions, the operator $\widetilde{K_{3}}$ and the bilinear form $\widetilde{B_2}$
satisfy the estimates%
\begin{align}
\left| \widetilde{B_2}(f,g)\right| & \leq C_{p}\left\Vert K(1)\right\Vert _{\BMO}\left\Vert f\right\Vert _{L^{p}_E(X)}\left\|g\right\|_{L^{p'}_{E'}(X)}\label{eq:K5}, \\
\left|\left< \widetilde{K_{3}}f,g\right>\right| & \leq C_{p}\left\Vert K^{*}(1)\right\Vert
_{\BMO}\left\Vert f\right\Vert _{L^{p}_E(X)}\left\|g\right\|_{L^{p'}_{E'}(X)}\label{eq:K3}.
\end{align}
Again, $p'=p/(p-1)$ is the conjugate exponent to $p$.
\end{theorem}


\begin{proof}
For the proof we use paraproduct operators which are formally given by
\begin{equation}
(P_{a}f)(x)=\sum_{n=0}^{\infty}\sum_{B\in\mathcal{A}_{-n}}\sum
_{R\in\mathcal{E}(B)}\frac{\left\langle a,d_{R}\right\rangle }{\mu
(B)}\left\langle d_{R},f\right\rangle 1_{B}(x),\label{eq:para}%
\end{equation}
where $a\in \BMO.$ Observe that $P_{a}$ is the linear extension of the
mapping%
\begin{align*}
d_{R}  & \longmapsto\left\langle a,d_{R}\right\rangle \frac{1_{B}}{\mu
(B)},\quad R\in\mathcal{E}(B)\quad\text{and}\\
1_{X}  & \longmapsto0,
\end{align*}
so that for finite linear combinations of Haar functions $f,g$ we have
\begin{equation}\label{eq:skalarident}
\left|\widetilde{B_2}(f,g)\right|=\left|\left<P_{K(1)}g,f\right>\right|\quad\text{and}\quad \left|\left<\widetilde{K_{3}}f,g\right>\right|=\left|\left<P_{K^*(1)}f,g\right>\right|.
\end{equation}
Now let both $f:X\rightarrow E$ and $a:X\rightarrow \mathbb{R}$ be finite linear combinations of Haar functions.
We then consider the bilinear operation%
\[
P(a,f):=\sum_{k}(\mathbb{E}_{k}f)(\Delta_{k+1}a),
\]
which has an immediate connection to a paraproduct operator, since we can
compute for $R\in\mathcal{E}(B)$%
\[
\left\langle P(a,f),d_{R}\right\rangle =\sum_{k}\left\langle (\mathbb{E}%
_{k}f)(\Delta_{k+1}a),d_{R}\right\rangle =\left\langle a,d_{R}\right\rangle
\frac{1}{\mu(B)}\int_{B}f~d\mu=\left\langle f,P_{a}d_{R}\right\rangle .
\]
Additionally, $\left\langle P(a,f),1\right\rangle =\left\langle f,P_{a}%
1\right\rangle ,$ since%
\begin{equation}
\left\langle P(a,f),1\right\rangle =\sum_{k}\mathbb{E[(E}_{k}f)\left(
\Delta_{k+1}a\right)  ]=0=\left\langle f,P_{a}1\right\rangle,\label{eq:const1}%
\end{equation}
so $P(a,\cdot)$ is the adjoint of $P_a$. Now we use a result that can be found in \cite{Figiel1990}, pages 108-109 and \cite{FigielWojtaszczyk2001}, page 593, which allows us to deduce the $L^p_E$--boundedness of the operator $P(a,\cdot)$ (note that we have a regular sequence of $\sigma$-algebras $\mathcal{F}^{\lev}_k)$ and the estimate
\begin{equation}
\left\Vert P(a,f)\right\Vert _{L_{E}^{p}(X)}\leq C\left\Vert a\right\Vert
_{BMO}\left\Vert f\right\Vert _{L_{E}^{p}(X)}%
\label{eq:Paf}%
\end{equation}
for $f\in L_{E}^{p}(X)$ and $a\in\BMO$.

With the $L^p_E$-boundedness of $P(a,\cdot)$, (\ref{eq:skalarident}) and the fact that $P_a$ is the adjoint of $P(a,\cdot)$, we finally get that
\[
\left|\widetilde{B_2}(f,g)\right|
=\left|\left<P(K(1),f),g\right>\right|\leq C\left\|K(1)\right\|_{\BMO}\left\|f\right\|_{L^p_E(X)}\left\|g\right\|_{L^{p'}_{E'}(X)}
\]
and
\[
\left|\left< \widetilde{K_{3}}f,g\right>\right|
=\left|\left<f,P(K^*(1),g)\right>\right|\leq C\left\|K^*(1)\right\|_{\BMO}\left\|f\right\|_{L^p_E(X)}\left\|g\right\|_{L^{p'}_{E'}(X)},
\]
since $E$ is a $\UMD$--space and thus reflexive.
\end{proof}

\begin{proof}[Proof of Theorem \ref{th:main}]
For $1/p+1/q=1$, let $f\in L^p_E(X)$ and $g\in L^{p'}_{E'}(X)$ be finite linear combinations of Haar functions, then we have
\[
\left<Kf,g\right>=\widetilde{B_2}(f,g)+\left<\widetilde{K_3}f,g\right>+\sum_{m=0}^\infty \left<k_{1,m},g\otimes f\right>+B_{2,m}(f,g)+\left<K_{3,m}f,g\right>,
\]
where these operators are defined in (\ref{eq:defb1b2}),(\ref{eq:defk2k3}),(\ref{eq:defk1m}),(\ref{eq:defb1m}) and (\ref{eq:defk2m}). 
Thus we obtain from Theorems \ref{th:k1} and \ref{th:k3} that there exists a constant $C_K$ which has only the stated dependences and we have
\[
|\left<Kf,g\right>|\leq C_K \left\|f\right\|_{L^p_E(X)}\left\|g\right\|_{L^{p'}_{E'}(X)}.
\]
Hence for fixed $f\in L^p_E(X)$ which is a finite linear combination of Haar functions, the functional $S_f$ defined by
\[
S_f:g\mapsto\left<Kf,g\right>
\]
is bounded on the subspace $U$ consisting of finite linear combinations of Haar functions of $L^{p'}_{E'}(X)$. Since $U$ is dense in $L^{p'}_{E'}(X)$, it has a unique continuous extension to the whole space $L^{p'}_{E'}(X)$. Recall now that $\UMD$-spaces are reflexive, and so $L^{p'}_{E'}(X)$ is canonically identified with $L^p_E(X)$. Hence there exists $z\in L^p_E(X)$ such that
\[
\left<z,g\right>=\left<Kf,g\right> \text{ for all } g\in U\quad\text{and}\quad \left\|z\right\|_{L^p_E(X)}\leq C_K \left\|f\right\|_{L^p_E(X)}.
\]
We get that $z=Kf$ since they have the same Haar coefficients, and so
\[
\left\|Kf\right\|_{L^p_E(X)}\leq C_K \left\|f\right\|_{L^p_E(X)}
\]
for all finite linear combinations of Haar functions $f$. Since again these functions are dense in $L^p_E(X)$, $K$ has a unique bounded linear extension to all of $L^p_E(X)$ and the theorem is proved.
\end{proof}

The rest of Section \ref{sec:main} is now devoted to proving Theorem \ref{th:k1}.

\subsection{Rearrangement and Shift Operators\label{sec:rearr}}

Definition \ref{def:rearr} (The definition of the injection $\tau$ on $\mathcal{A}_{m,i}$) gives rise to rearrangement and shift operators,
which are closely related to the integral operators $K_1,K_3$ and the bilinear form $B_2$. For $m\in\mathbb{N}_{0},$ $1\leq i\leq M_1q^m$ (see Proposition \ref{prop:boundednumber}) we define 
\begin{align}
U_{m,i}^{(k)}(f):=  & \sum_{A\in\mathcal{A}_{m,i}}\frac{\left\langle d_{Q_k(\tau(A))},f\right\rangle}{\sqrt{\mu(A)}} \left(  1_{A}-\frac
{\mu(A)}{\mu(\tau(A))}1_{\tau(A)}\right) \label{eq:rearr2}\\
T_{m,i}^{(j,k)}(f):=  & \sum_{A\in\mathcal{A}_{m,i}}\left\langle d_{Q_k(\tau(A))},f\right\rangle d_{Q_{j}%
(A)},\label{eq:rearr4}
\end{align}
where $f$ is a finite linear combination of Haar functions and $Q_{j}(A)$ is any enumeration of the elements in $\mathcal{E}(A)$. If the parameter $k$ is greater than the number $N(A)$ of Haar functions corresponding to children of $A$, we simply set $d_{Q_k(A)}\equiv 0$.

\begin{remark}
We see that $U_{m,i}^{(k)}$ is the linear extension of the map
\[
d_{Q_k(\tau(A))}\longmapsto\frac{1}{\sqrt{\mu(A)}}\left(1_{A}-\frac{\mu(A)}{\mu(\tau(A))}1_{\tau(A)}\right),\quad\text{for }1\leq k\leq N(\tau(A)),
\]
with $A\in\mathcal{A}_{m,i}.$ Analogously the mapping $T_{m,i}^{(j,k)}(f) $
is the linear extension of%
\[
d_{Q_k(\tau(A))}\longmapsto d_{Q_{j}(A)},\quad\text{for }1\leq k\leq N(\tau(A)),
\]
where $A\in\mathcal{A}_{m,i}.$
\end{remark}\vspace{0.5cm}

In order to show Theorem \ref{th:k1}, we prove the following
$L^{p}$-bounds of the operators $U_{m,i}^{(k)}$ and $T_{m,i}^{(j,k)}$:

\begin{proposition}
\label{prop:shift}The operators $U_{m,i}^{(k)}$ and $T_{m,i}^{(j,k)}$  satisfy the $L^{p}_E%
(X)$-estimate $(1<p<\infty)$%
\begin{align}
\left\Vert U_{m,i}^{(k)}:L^{p}_E(X)\rightarrow L^{p}_E(X)\right\Vert &\leq C_{p}(m+1),\\
\left\Vert T_{m,i}^{(j,k)}:L^{p}_E(X)\rightarrow L^{p}_E(X)\right\Vert &\leq C_{p}(m+1).
\end{align}
for all $1\leq j,k\leq N-1$, where $C_p$ depends only on $p,X$ and the $\UMD$--constant of $E$. Here, as in Section \ref{sec:homtype}, $N$ is the maximal number of children a dyadic cube can have.
\end{proposition}

The rough idea of the proof of these bounds is the following: We prove
a version of Proposition \ref{prop:shift} under the
constraint that we restrict the sum in (\ref{eq:rearr2}) and
(\ref{eq:rearr4}) from $\mathcal{A}_{m,i}$ to a collection that satisfies the so
called Figiel's compatibility condition. In this case we get a bound, which is
independent of $m.$ Thereafter we invoke the decomposition of $\mathcal{A}_{m,i}$ into such subcollections introduced in Section \ref{sec:sep}. 

\subsection{Figiel's Compatibility Condition\label{sec:figiel}}
Here we review the martingale estimates of rearrangement operators that satisfy Figiel's compatibility condition. We follow \cite{Figiel1988} and the expositions \cite{FigielWojtaszczyk2001}, \cite{Mueller2005}.
\begin{definition}
Let $\mathcal{D}\subseteq\mathcal{A}_{m,i}$ be a subset of $\mathcal{A}_{m,i}$
and $\tau:\mathcal{D}\rightarrow\mathcal{A}$ be an injective map. We say that
the pair $(\tau,\mathcal{D)}$ satisfies \emph{Figiel's} \emph{compatibility
condition} if the collection
\[
\mathcal{Z}:=\{A,\tau(A),A\cup\tau(A):A\in\mathcal{D}\}
\]
is nested, $\lev A=\lev \tau(A)$ and $\tau(A)\notin\mathcal{D}$ for all $A\in\mathcal{D}$ 
\end{definition}
Recall that a collection of sets $\mathcal{Z}$ is said to be
\emph{nested}, if for every choice $A,B\in\mathcal{Z}$ it holds that either%
\[
A\subseteq B\quad\text{or\quad}B\subseteq A\quad\text{or\quad}A\cap
B=\emptyset.
\]

We remark that if $(\tau,\mathcal{D})$ satisfies Figiel's compatibility
condition, the pair $(\tau^{-1},\tau(\mathcal{D}))$ also satisfies Figiel's
compatibility condition. Then the following theorems concerning the boundedness of the operators
$T$ and $U$ hold.

\begin{theorem}
\label{th:figiel1}Let $(\tau,\mathcal{D})$ satisfy the compatibility
condition. Then the operator
\[
T_{m,i}^{(j,k),\mathcal{D}}f:=\sum_{A\in\mathcal{D}}\left\langle d_{Q_k(\tau(A))},f\right\rangle d_{Q_{j}(A)}%
\]
is bounded in $L^{p}$ for all $1\leq j,k\leq N-1$ and satisfies the estimate%
\[
\left\Vert T_{m,i}^{(j,k),\mathcal{D}}f\right\Vert _{L^{p}_E(X)}\leq
C_{p}\left\Vert f\right\Vert _{L^{p}_E(X)}%
\]
where $C_{p}$ depends only on $p,X$ and the $\UMD$--constant of $E$.
\end{theorem}

\begin{theorem}
\label{th:figiel2}Let $(\tau,\mathcal{D})$ satisfy the compatibility
condition. Then the operator
\[
U_{m,i}^{(k),\mathcal{D}}f:=\sum_{A\in\mathcal{D}}\frac{\left\langle d_{Q_k(\tau(A))},f\right\rangle}{\sqrt{\mu (A)}} \left(  1_{A}-\frac
{\mu(A)}{\mu(\tau(A))}1_{\tau(A)}\right)
\]
is bounded in $L^{p}$ and satisfies the estimate%
\[
\left\Vert U_{m,i}^{(k),\mathcal{D}}f\right\Vert _{L^{p}_E(X)}\leq C_{p}\left\Vert
f\right\Vert _{L^{p}_E(X)},
\]
where $C_{p}$ depends only on $p,X$ and the $\UMD$--constant of $E$.
\end{theorem}

The proofs of the foregoing two theorems are slight modifications of the analogous results for the Haar system in $\mathbb{R}$ in \cite{Mueller1995}.

\begin{remark}
If we apply this theorem to the collection $\tau(\mathcal{D})$ and the map
$\tau^{-1},$ we get that the operator%
\[
f\mapsto W_{m,i}^{(k),\mathcal{D}}(f):=\sum_{A\in\mathcal{D}}\frac{\left\langle d_{Q_k(A)},f\right\rangle}{\sqrt{\mu(\tau(A))}}\left(
1_{\tau(A)}-\frac{\mu(\tau(A))}{\mu(A)}1_{A}\right)
\]
is bounded on $L^{p}_E.$ $W_{m,i}^{\mathcal{D}}$ is the linear extension of the
mapping%
\[
d_{Q_k(A)}\longmapsto\frac{1}{\sqrt{\mu(\tau(A))}}\left(1_{\tau(A)}-\frac{\mu(\tau(A))}{\mu(A)}1_{A}\right),
\]
for $A\in\mathcal{D}.$
\end{remark}

\subsection{The Boundedness of the Operators
$W_{m,i}^{(k)},U_{m,i}^{(k)},T_{m,i}^{(j,k)}$}\label{sec:boundwithoutfigiel}
Using the decomposition theorems proved in Section \ref{sec:homtype}, we are now able to reduce the general case of Proposition \ref{prop:shift} to the special case of nested collections proved in the preceding Section \ref{sec:figiel}. \\ 
\begin{proof}[Proof of Proposition \ref{prop:shift}]
If we invoke the decomposition results of Chapter \ref{sec:homtype}, we see that $\mathcal{C}_{m,i}$ splits into $M$ collections $\mathcal{G}$, where $M$ is
constant. Further, every $\mathcal{A}_{\mathcal{G}}$ splits into $L(m+1)$ collections
$\mathcal{L}.$ $\mathcal{L}$ decomposes in $\mathcal{H},\mathcal{I}%
_{1},\mathcal{I}_{2},\mathcal{J}_{1},\mathcal{J}_{2},\mathcal{J}_{3},$ where
on $\mathcal{H},$ $\mathcal{I}_{1}$ and $\mathcal{I}_{2},$ the operators
$W_{m,i}^{(k)},U_{m,i}^{(k)}$ and $T_{m,i}^{(j,k)}$ are bounded by a constant which is independent of $m$. Since the collections  $\mathcal{H},$ $\mathcal{I}_{1}$ and $\mathcal{I}_{2}$ satisfy Figiel's compatibility condition (with the injection $\tau$) by Theorem \ref{th:zerl}, this follows directly from Theorems \ref{th:figiel1},
\ref{th:figiel2} and the Remark after them.
The collections $\mathcal{J}_{i},$ $i\in\{1,2,3\}$ need further arguments. For the following we fix an index $1\leq j\leq N-1$ and define the following map on $\gamma_{1}(\mathcal{J}_{1})$ 
\begin{align*}
\rho:\gamma_{1}(\mathcal{J}_{1})  & \rightarrow\mathcal{A}\\
A  & \longmapsto\tau\circ\gamma_{1}(A).
\end{align*}
Since the mapping $(\gamma_{1},\mathcal{J}_{1})$ (and hence also $(\gamma
_{1}^{-1},\gamma_{1}(\mathcal{J}_{1})))$ satisfies Figiel's compatibility
condition (note Lemma \ref{lem:colJ1}), we see from Theorem \ref{th:figiel1} and Theorem \ref{th:figiel2}
that the linear extensions of the mappings%
\[
T_{\gamma_{1}}^{(j,k)}:d_{Q_k(A)}\longmapsto d_{Q_{j}(\gamma_{1}(A))},\quad W_{\gamma_{1}%
}^{(k)}:d_{Q_k(A)}\longmapsto\frac{1}{\sqrt{\mu(\gamma_1(A))}}\left(1_{\gamma_{1}(A)}-\frac{\mu(\gamma_{1}(A))}{\mu(A)}1_{A}\right)%
\]
where $A\in\mathcal{J}_{1},$ are bounded on $L^{p} $. 
In Theorem \ref{th:zerl} we constructed a decomposition of $\mathcal{I}$ into $\mathcal{I}%
_{1}$ and $\mathcal{I}_{2}$ which both satisfied Figiel's compatibility condition with the injection $\tau$. Since with $\gamma_1$ we moved sets in $\mathcal{J}_1$ away from the boundary of their arithmetic predecessors, we are in the same position for the collection $\gamma_1(\mathcal{J}_1)$ and the injection $\rho$, since if we again perform a decomposition like in Proposition \ref{prop:deccmi} we are able to use the Remark after the proof of Lemma \ref{lem:schoenI}. We thus obtain from the $\mathcal{I}$-part of Theorem \ref{th:zerl} and again from Theorems \ref{th:figiel1},
\ref{th:figiel2} and the Remark after them that the linear extension of the mappings%
\[
T_{\rho}^{(j,k)}:d_{Q_k(A)}\longmapsto d_{Q_{j}(\rho(A))},\quad W_{\rho}^{(k)}:d_{Q_k(A)}%
\longmapsto\frac{1}{\sqrt{\mu(\rho(A))}}\left(1_{\rho(A)}-\frac{\mu(\rho(A))}{\mu(A)}1_{A}\right)%
\]
where $A\in\gamma_{1}(\mathcal{J}_{1}),$ are bounded
on $L^{p}_E$ by a constant which depends only on $X$. For the same reason, we may even replace $\rho$ by $\rho^{-1}$ and the
assertions stay valid. We conclude that the composition%
\[
T_{\gamma_{1}}^{(j,1)}\circ T_{\rho^{-1}}^{(1,k)}=:T_{\tau^{-1}}^{(j,k)}%
\]
is bounded on $L^p_E$ and it is the linear extension of the map
\[
d_{Q_k(\tau(A))}\longmapsto d_{Q_{j}(A)},
\]
where $A\in\mathcal{J}_{1}.$  We remark that $T_{\tau^{-1}}^{(j,k)}$
is the shift operator $T_{m,i}^{(j,k),\text{ }\mathcal{J}_{1}},$ which is thus shown to be bounded. Now we
come to the linear extension of the map%
\[
d_{Q_k(A)}\longmapsto\frac{1}{\sqrt{\mu(\tau(A))}}\left(1_{\tau(A)}-\frac{\mu(\tau(A))}{\mu(A)}1_{A}\right),\quad\text{for
}A\in\mathcal{J}_{1}%
\]
which is the mapping $W_{m,i}^{(k),\mathcal{J}_{1}}.$ For
finite linear combinations of Haar functions $f=\sum_{l}a_{l}d_{Q_{l}},$ where $Q_{l}%
=Q_k(A_{l})$,$A_{l}\in\mathcal{J}_{1}$ and $a_l\in E$, $W_{m,i}^{(k),\mathcal{J}_{1}}$ has the representation
\[
W_{m,i}^{(k),\mathcal{J}_{1}}f=(W_{\rho}^{(1)}\circ T_{\gamma_{1}}^{(1,k)})(f)+\sum a_{Q_{l}%
}\sqrt{\frac{\mu(\tau(A_{l}))}{\mu(\gamma_{1}(A_{l}))}}W_{\gamma_{1}}^{(k)}(d_{Q_{l}}).
\]
With the unconditionality of the $\{d_{Q}\}$ and Kahane's contraction
principle, we conclude that $W_{m,i}^{(k),\mathcal{J}_{1}}$ is bounded on
$L^{p}_E.$ Analogously, for $f=\sum_{l}a_{l}d_{Q_{l}},$ where $Q_{l}%
=Q_k(\tau(A_{l}))$ and $A_{l}\in\mathcal{J}_{1},$ we have the
representation%
\[
U_{m,i}^{(k),\mathcal{J}_{1}}f=\sum_{l}a_{l}\sqrt{\frac{\mu(A_{l})}{\mu(\gamma
_{1}(A_{l}))}}W_{\rho^{-1}}^{(k)}(d_{Q_{l}})+(W_{\gamma_{1}^{-1}}^{(1)}\circ T_{\rho^{-1}}^{(1,k)})(f).
\]
A similar reasoning applies to $\mathcal{J}_{2},$ where in this case we let%
\[
\gamma_{2}:\tau(\mathcal{J}_{2})\rightarrow\mathcal{A}%
\]
to move away from the boundary of the arithmetic predecessor. The mapping $\rho$ is
defined as%
\begin{eqnarray*}
\rho:\mathcal{J}_{2}  & \rightarrow & \mathcal{A}\\
A  & \mapsto & \gamma_{2}\circ\tau(A).
\end{eqnarray*}
In the case for $\mathcal{J}_{3}$ we define both injections
$\gamma_{1}$ and $\gamma_{2}$ from above to act on $\mathcal{J}_{3}$ and
$\tau(\mathcal{J}_{3})$ respectively and set%
\begin{align*}
\rho:\gamma_{1}(\mathcal{J}_{3})  & \rightarrow\mathcal{A}\\
A  & \longmapsto\gamma_{2}\circ\tau\circ\gamma_{1}(A).
\end{align*}
If we summarize these considerations, we get a decomposition of the operators $W_{m,i}^{(k)},U_{m,i}^{(k)}$ and $T_{m,i}^{(j,k)}$ into a sum of $C(m+1)$ bounded operators on $L^p$, where their bound depends only on $p,X$ and the $\UMD$--constant of $E$. Since so does $C$, we get the assertions of Proposition \ref{prop:shift} and that $W_{m,i}$ is bounded on $L^p$ by $C_p(m+1)$. 
\end{proof}

\subsection{The Proof of Theorem \ref{th:k1}}
\newcounter{mycount1} \setcounter{mycount1}{\value{theorem}} \setcounter{theorem}{\themycount}

\begin{theorem}
For all $1<p<\infty$ there exists a constant $C_{p}$ depending only on $p,X$, the $\UMD$ constant of $E$ and $C_S$ from (\ref{eq:struct1}), such that for all $f\in L^{p}_E(X),g\in L^{p'}_{E'}(X)$, which are finite linear combinations of Haar functions, the operators $K_{1,m},K_{3,m}$ and the bilinear form $B_{2,m}$ satisfy the following estimates:
\begin{align}
\left\Vert K_{1,m}(f)\right\Vert _{L^{p}_E(X)} & \leq C_{p}(m+1)q^{-m\delta}\left\Vert
f\right\Vert _{L^{p}_E(X)}, \tag{\ref{eq:K1}}\\
\left\Vert K_{3,m}(f)\right\Vert _{L^{p}_E(X)} & \leq C_{p}(m+1)q^{-m\delta}\left\Vert
f\right\Vert _{L^{p}_E(X)}, \tag{\ref{eq:K2}}\\
\left|B_{2,m}(f,g)\right| & \leq C_{p}(m+1)q^{-m\delta}\left\Vert
f\right\Vert _{L^{p}_E(X)}\left\|g\right\|_{L^{p'}_{E'}(X)}.\tag{\ref{eq:K4}}
\end{align}
Here, $p'=p/(p-1)$ denotes the conjugate exponent to $p$.
\end{theorem}


\begin{proof}
It holds that 
\[
K_{1,m}=\sum_{i=1}^{M_1 q^m}\sum_{j,k=1}^{N-1} T_{m,i}^{(j,k)}\circ\mathcal{M}_{m,i}^{(j,k)},
\]
where $T_{m,i}^{(j,k)}$ is the shift operator introduced in (\ref{eq:rearr4}) and $\mathcal{M}_{m,i}^{(j,k)}$ is the Haar multiplication operator which maps
\[
d_{Q_k(\tau(A))}\mapsto \left<k,d_{Q_j(A)}\otimes d_{Q_k(\tau(A))}\right>d_{Q_k(\tau(A))}\quad\text{for }A\in\mathcal{A}_{m,i}.
\]
Analogously we get
\begin{equation}
K_{3,m}=\sum_{i=1}^{M_1 q^m}\sum_{k=1}^{N-1} U_{m,i}^{(k)}\circ \mathcal{M}_{m,i}^{(k)},\quad B_{2,m}(f,g)=\left<\sum_{i=1}^{M_1q^m}\sum_{j=1}^{N-1} W_{m,i}^{(j)}\circ \widetilde{\mathcal{M}}_{m,i}^{(j)}f,g\right>,\label{eq:decompK}
\end{equation}
where $\mathcal{M}_{m,i}^{(k)}$ and $\widetilde{\mathcal{M}}_{m,i}^{(j)}$ are Haar multiplication operators which map
\[
d_{Q_k(\tau(A))}\mapsto \frac{\left<k,1_A\otimes d_{Q_k(\tau(A))}\right>}{\sqrt{\mu(A)}}d_{Q_k(\tau(A))}\quad\text{and}\quad d_{Q_j(A)}\mapsto\frac{\left<k,d_{Q_j(A)}\otimes 1_{\tau(A)}\right>}{\sqrt{\mu(\tau(A))}}d_{Q_j(A)},
\]
respectively. These decompositions follow from the definition of $K_{1,m},K_{3,m},B_{2,m}$ in $(\ref{eq:defk1m})-(\ref{eq:defk2m})$, Proposition \ref{prop:boundednumber}, the definition of the shifts $T_{m,i}^{(j,k)}$ and rearrangements $U_{m,i}^{(k)},W_{m,i}^{(j)}$ in (\ref{eq:rearr2}), (\ref{eq:rearr4}) and the remark after Theorem \ref{th:figiel2}. Since Haar multipliers are bounded on $L^p_E(X)$ by the supremum of their coefficients, we deduce by the structural estimate (\ref{eq:struct1}) and Proposition \ref{prop:shift}
\begin{eqnarray*}
\left\|K_{1,m}f\right\|_{L^p_E(X)}&=&\left\|\sum_{i=1}^{M_1 q^m}\sum_{j,k=1}^{N-1}T_{m,i}^{(j,k)}\circ\mathcal{M}_{m,i}^{(j,k)}f\right\|_{L^p_E(X)} \\
&\leq& \sum_{i=1}^{M_1 q^m}\sum_{j,k=1}^{N-1}\left\|T_{m,i}^{(j,k)}\right\|_{L^p_E(X)\rightarrow L^p_E(X)}\left\|\mathcal{M}_{m,i}^{(j,k)}\right\|_{L^p_E(X)\rightarrow L^p_E(X)}\left\|f\right\|_{L^p_E(X)}\\
&\leq& C q^m(m+1)\sup_{\stackrel{1\leq j,k\leq N-1}{A\in\mathcal{A}_{m,i}}}\left|\left<k,d_{Q_j(A)}\otimes d_{Q_k(\tau(A))}\right>\right|\left\|f\right\|_{L^p_E(X)}\\
&\leq & C q^m(m+1) \sup_{A\in\mathcal{A}_{m,i}}\left(1+\frac{d(A,\tau(A))}{q^{\lev A}}\right)^{-1-\delta}\left\|f\right\|_{L^p_E(X)},
\end{eqnarray*}
where $C$ is a constant only depending on $p,X,$ the $\UMD$-constant of $E$ and $C_S$ that possibly changes from line to line. Now we get from the definition of $\mathcal{C}_m$ in Section \ref{sec:dyadicannuli} (and hence from the corresponding property for $\mathcal{A}_{m,i}$) that the last expression in the previous display is less or equal
\[
Cq^{-\delta m}(m+1)\left\|f\right\|_{L^p_E(X)},
\]
which is the required conclusion for $K_{1,m}$. The two remaining assertions follow from similar arguments using the decompositions in (\ref{eq:decompK})
\end{proof}
\setcounter{theorem}{\value{mycount1}}

\paragraph{Closing Remark.} 
For singular integral operators in the scalar valued case, the $T(1)$ theorem was extended in an important series of papers (starting with the pioneering contribution in \cite{DavidMattila2000} and \cite{David1998} and extended in \cite{NazarovTreilVolberg2003}) to metric measure spaces that are not necessarily of homogeneous type. 
In this nonhomogeneous setting there holds a $\UMD$ valued $T(1)$ theorem (\cite{Hytonen2009}). To differentiate those results from our work in this paper we note that the initial step in \cite{David1998,DavidMattila2000,NazarovTreilVolberg2003,Hytonen2009} is the expansion of the singular integral kernel along the \emph{anisotropic tensor product} Haar basis in $\mathbb{R}^n\times\mathbb{R}^n$. Such an expansion leads to decompositions of integral operators that are \emph{structurally} different from the basic building blocks studied in the present paper.

As mentioned already in the introduction, our decomposition (\ref{eq:decomp2}) into simple rearrangements and shifts is the central assertion of our work and it permits us to study integral operators beyond the Calderón-Zygmund class (\cite{KamontMueller2006,LeeMuellerMueller2011}).

\paragraph{Acknoledgements}
Both authors are supported by FWF P 20166-N18. This work is part of the PhD thesis of the second named author at the Department of Analysis at JKU Linz.

\nocite{Christ1990}\nocite{Christ1990book}\nocite{CoifmanWeiss1977}%
\nocite{Mueller2005}\nocite{JohnNirenberg1961}

\nocite{MaciasSegovia1979}\nocite{Chatterji1968}\nocite{Bourgain1983}%
\nocite{Garsia1973}\nocite{DavidJourne1984}

\nocite{Figiel1990}\nocite{Xu1995}\nocite{MaciasSegovia1979Distr}%
\nocite{BlascoXu1991}\nocite{JansonJones1982}

\nocite{FigielWojtaszczyk2001}\nocite{Blasco1989}\nocite{Figiel1988}%
\nocite{Stein1970}\nocite{Stein1993}

\nocite{MarcusPisier1981}\nocite{Figiel1989}

\nocite{BennettSharpley1988}\nocite{Burkholder2001}%
\nocite{AimarBernardisIaffei2007}
\nocite{CoifmanMaggioni2006}
\nocite{Kahane1985}
\nocite{HanSawyer1994}
\nocite{Hytonen2009}
\nocite{BeylkinCoifmanRokhlin1992}
\nocite{NazarovTreilVolberg2003}
\nocite{David1991}
\nocite{David1998}
\nocite{Mueller1995}
\nocite{MuellerSchechtman1991}
\nocite{DavidMattila2000}
\nocite{HytonenMcIntoshPortal2008}
\nocite{CoifmanJonesSemmes1989}
\nocite{LeeMuellerMueller2011}
\nocite{KamontMueller2006}

\bibliographystyle{plain}
\bibliography{Passenbr}

\end{document}